\documentclass[12pt]{amsart}
\usepackage{latexsym,fancyhdr,amssymb,color,amsmath,amsthm,graphicx,listings,comment}
 
\newtheorem{thm}{Theorem}       \newtheorem{propo}{Proposition}
       \newtheorem{coro}{Corollary}
\usepackage[section]{placeins}      \let\paragraph\subsection
\newcommand\scalemath[2]{\scalebox{#1}{\mbox{\ensuremath{\displaystyle #2}}}}

\title{Energized simplicial complexes}
\author{Oliver Knill} \date{8/18/2019}
\address{Department of Mathematics \\ Harvard University \\ Cambridge, MA, 02138 }

\begin{document}

\begin{abstract}
For a simplicial complex with $n$ sets, let $W^-(x)$ 
be the set of sets in $G$ contained in $x$ and $W^+(x)$
the  set of sets in $G$ containing $x$. A function $h: G \to \mathbb{Z}$ defines 
for every $A \subset G$ an energy $E[A]=\sum_{x \in A} h(x)$. The function energizes
the geometry similarly as divisors do in the continuum, where the Riemann-Roch quantity
$\chi(G)+{\rm deg}(D)$ plays the role of the energy.
Define the $n \times n$ matrices $L=L^{--}(x,y)=E[W^-(x) \cap W^-(y)]$ and 
$L^{++}(x,y) = E[W^+(x) \cap W^+(y)]$ as well as $L^{+-}(x,y) = E[W^+(x)  W^-(y)]=(L^{-+})^T$.
With the notation $S(x,y)=1_n \omega(x) =\delta(x,y) (-1)^{{\rm dim}(x)}$ and ${\rm str}(A)={\rm tr}(SA)$ 
define $g=S L^{++} S$. The results are: ${\rm det}(L)={\rm det}(g) = \prod_{x \in G} h(x)$ and 
$E[G] = \sum_{x,y} g(x,y)$ and $E[G]={\rm str}(g)$. The number of positive eigenvalues of 
$g$ is equal to the number of positive energy values of $h$. In special cases, more is true:
A) If $h(x) \in \{ -1, 1 \}$, all four homoclinic and heteroclinic 
matrices $L^{\pm \pm}$ are unimodular and $L^{-1} = g$, even if $G$ is a set of sets.
B) In the constant energy $h(x)=1$ case, $L$ and $g$ are isospectral, positive definite matrices
in $SL(n,\mathbb{Z})$ \cite{CountingMatrix}. For any set of sets $G$ we get so isospectral multi-graphs
defined by adjacency matrices $L^{++}$ or $L^{--}$ which have identical spectral or Ihara zeta
function. The positive definiteness holds for effective 
divisors $h>0$ in general. C) In the topological case $h(x)=\omega(x)$, the energy 
$E[G]={\rm str}(L) = {\rm str}(g) = \sum_{x,y} g(x,y)=\chi(G)$ is the Euler characteristic 
of $G$ and $\phi(G)=\prod_x \omega(x)$ \cite{EnergyComplex}, a product identity which holds for arbitrary set of
sets. D) For $h(x)=t^{|x|}$ with some parameter $t$ we have $E[H]=1-f_H(t)$ with 
$f_H(t)=1+f_0 t + \cdots + f_d t^{d+1}$ for the $f$-vector of $H$ and $L(x,y) = (1-f_{W^-(x) \cap W^-(y)}(t))$ and 
$g(x,y)=\omega(x) \omega(y) (1-f_{W^+(x) \cap W^+(y)}(t))$. Now, the inverse of $g$ is 
$g^{-1}(x,y) = 1-f_{W^-(x) \cap W^-(y)}(t)/t^{{\rm dim}(x \cap y)}$ and $E[G] = 1-f_G(t)=\sum_{x,y} g(x,y)$.
\end{abstract} 

\maketitle

\section{Introduction}

\paragraph{}
For a finite set $G$ of sets, the connection matrix $L$ is defined by $L(x,y)=1$ if $x$ and $y$ 
intersect and $L(x,y)=0$ if not. If $G$ is a finite abstract simplicial complex,
$L$ is unimodular. We look here at the case where $L(x,y)$ is an ``energy" of $x \cap y$. 
In the connection matrix case, the energy is the Euler characteristic of $x \cap y$.

\paragraph{}
Energizing a geometry $G$ using a function $h(x)$
is not only motivated by physics or probability density functions $h(x)$,
it also produces an affinity to algebraic geometry, as an integer-valued function with finite 
energy on an algebraic curve is a divisor. 
A combinatorial application is the construction of a large sample of isospectral
multi-graphs and related, the construction of isospectral, positive definite,
non-negative symmetric matrices $L^{++},L^{--}$ in $SL(n,Z)$. A set of sets $G$
defines a quadratic form, unimodular integral lattices
in $\mathbb{R}^n$ and isometric inner products with isometric ellipsoids
$(x,L^{++} x)=1, (x,L^{--} x)=1$.

\paragraph{}
Given a set of sets $G$, an energy function $h:G \to \mathbb{R}$ assigns an energy $\sum_{x \in A} h(x)$ 
to subsets $A$ of $G$ and in particular defines the total energy $E[G]=\sum_{x \in G} h(x)$. 
If $G$ is a simplicial complex and $h(x)=\omega(x)=(-1)^{{\rm dim}(x)}$, the matrix $L$ is 
unimodular \cite{EnergyComplex}. In this case, the energy $E[A]$ is the
Euler characteristic of a sub complex $A$ of $G$. We generalize this here first to $\{-1,1\}$-valued 
functions $h$. By allowing $h$ to be arbitrary real numbers
and also relaxing $G$ to be an arbitrary set of sets, we explore
also boundary cases, where things start to be different. 

\paragraph{}
The subject ties in with a variety of other topics \cite{AmazingWorld}. 
The unimodularity result has relations with hyperbolic dynamical systems, 
spectral theory, quadratic forms, integral lattices, zeta functions and the representation 
theory of graph arithmetic. Bringing in the discrete divisors in the form of 
energy relates to some discrete algebraic geometry: 
Riemann-Roch in the form of Baker-Norine theory \cite{BakerNorine2007} 
generalized to multigraphs in \cite{JamesMiranda2013} 
tells that the total energy $E[G]$ of a divisor is a signed distance $l(G)-l(K-G)$ 
with canonical divisor $K$ and where $l(G)$ is the least
energy value one has to add so that the structure can be made effective  
(non-negative everywhere) modulo changing the energy using principal divisors. 

\paragraph{}
We comment more on the Riemann-Roch situation elsewhere but for now just want to 
mention that in the discrete, Riemann-Roch is very concrete. 
While in the continuum, adding a principal divisor is done via rational
functions, the principal divisors in the discrete are just the image of a discrete Laplacian.
Adding such a divisor is physical in the energy picture: it means taking 
the energy on one node of the geometry and distributing the energy locally 
using the Laplacian given by the incidence graph of $G$, an operation which is energy preserving.

\paragraph{}
Also in the energized situation, the matrix entries $L$ and its inverse $g$ relate to
a hyperbolic dynamical system in which stable and unstable manifolds intersect 
in a ``heteroclinic tangle". A bit surprisingly, this structure appears 
in the simplest possible frame work of mathematics, where one just takes a finite set of 
sets. One does not even need to have a simplicial complex. But having the later structure
makes things nicer: the gradient system to the dimension functional is then Morse 
in the sense that every $x \in G$ is a hyperbolic critical point with index 
$\omega(x) \in \{-1,1\}$. The stable and unstable manifolds $W^-(x), W^+(x)$ 
are localized and the energy of their intersections define the matrices. 

\paragraph{}
As for $\{-1,1\}$-valued energies, one has unimodularity so that $g=L^{-1}$ is integer valued.
This includes not only the original topological frame work $h(x)=\omega(x)$ 
but also in a situation close to quantum physics:
in the manifold case, where the kernels $g$ of the inverse of 
Laplacians are singular and the potential $V_x(y) = g(x,y)$ is long range; now, in the discrete, 
the interaction is finite range and can be of service in relativistic frame works. 
We remain here in a elementary combinatorial and linear algebra set-up. Much even 
generalizes to sets of sets $G$, a situation obtained by dropping the only axiom.

\paragraph{}
We still should stress that having a mass gap, an interval around $0$ without 
spectrum of the Laplacian is rather special as it also can survive the van Hove limit 
in which we look at an infinite structure like $\mathbb{Z}^d$. This is not automatically 
a consequence of a discretization. Indeed, the most obvious discrete Laplacian $H$ in mathematics 
does not have this property: if $d$ is the exterior derivative defined on $G$
and if $H=d d^* + d^* d=(d+d^*)^2$ is the Hodge Laplacian on the set of differential forms, 
then $H = \oplus_k H_k$  is a $n \times n$ matrix which 
is block diagonal and the nullity of $H_k$ is the Betti number $b_k$ of $G$ by Hodge theory.
Even when inverting $H_k$ on the ortho-complement of the harmonic forms, the pseudo inverse 
matrix entries $H_k^{-1}(x,y)$ are long-range, meaning that two points $x,y$ in general interact, 
even if they are arbitrarily far apart.

\paragraph{}
Beside the situation where $x$ has an energy $h(x)=\omega(x)$, we have also seen the case 
where the energy is constant $1$. In this ``effective divisor" case, the matrices $L$ and $g$ are in 
$SL(n,\mathbb{Z})$ and are positive definite as well as isospectral. It is a bit surprising that 
the result still holds also in the case when $G$ is just a set of sets. The 
matrices $L^{--}$ and $L^{++}$ are then integer valued symmetric matrices which are isospectral. 
Each of them can be seen as the adjacency matrix of a multi-graph $\Gamma^{++}$ and $\Gamma^{--}$. 
The fact that the matrices are isospectral means than that the graphs are isospectral. 
There are not many general tools to generate isospectral geometries, one is by Sunada
(like for example used in \cite{HalbeisenHungerbuehler}) which is relevant in 
structural chemistry \cite{RandicNovicPlavsic}.

\section{The inverse}

\paragraph{}
A finite abstract simplicial complex
is a set of finite non-empty sets closed under the operation of taking finite non-empty subsets.
For $x \in G$, the star $W^+(x)$ of $x$ is the set of simplices which contain $x$ (including $x$).
The core $W^-(x)$ is the set of simplices contained in $x$ (including $x$).
Define $\omega(x) = (-1)^{{\rm dim}(x)}$, where ${\rm dim}(x) = |x|-1$ and $|x|$ is the 
cardinality of $x$. 

\paragraph{}
Define the matrix $S(x,y) = \delta(x,y) \omega(x)$ which has as trace the Euler characteristic
$\chi(G)$ and which defines a super trace ${\rm str}(A) = {\rm tr}(SA)$ for any $n \times n$ matrix $A$.
A function $h: G \to \mathcal R$ defines an energy of subsets $A$ of $G$:
$$ E[A] = \sum_{x \in A} h(x) \; . $$

\paragraph{}
Define the $n \times n$ matrices defined by homoclinic connections
$$  L^{--}(x,y) = E[ W^-(x) \cap W^-(y)],  L^{++}(x,y) = E[ W^+(x) \cap W^+(y)] \; . $$
For completeness, one can also define the heteroclinic connection matrices
$$  L^{+-}(x,y) = E[ W^+(x) \cap W^-(x)],  L^{-+}(x,y) = E[ W^-(x) \cap W^+(x)]  \;  $$
even so do not need them here. The later are upper and lower triangular if the simplicial 
complex $G$ is ordered from smaller to larger dimensional simplices. 
To be closer to the previous covered special cases 
\cite{EnergyComplex,CountingMatrix}
we can define $L=L^{--}$ and $g=S L^{++} S$, which is conjugated to $L^{++}$ as $S=S^{-1}$.
All these matrices are integer-valued if $h$ is integer-valued.

\paragraph{}
For the next result, we assume that $G$ is sorted so that if $x \subset y$, then $x$ comes 
before $y$ in the listing of the sets in $G$. An order of $G$ together defines the 
basis in which the matrices are written. 

\begin{thm}
The product $Lg$ is a lower triangular matrix with diagonal entries $h(x)^2$.
The entries of $Lg$ below the diagonal are all of the form $h(u)^2 - h(v)^2$.
\end{thm}

\begin{proof}
a) To check the diagonal entries, we have to verify
$$  \sum_y L^{--}(x,y) \omega(y) \omega(x) L^{++}(y,x) = h(x)^2 \; .  $$
Since $L^{--}(x,y)$ is the energy of all the sets in $x \cap y$
$L^{++}(y,x)$ is the energy of all sets containing $x \cup y$
the contribution of $y=x$ is $h(x)^2$. The claim therefore is
$$  \sum_{y \neq x} L^{--}(x,y) \omega(y) \omega(x) L^{++}(y,x) = 0 \; .  $$
As any set $u$ contributing to $L(x,y)$ must be strictly inside $x$
and any set $u$ contributing to $g(y,x)$ must contain $x$ strictly, 
the statement follows.  \\
b) If $x \subset z$ but not $x= z$, the upper triangular entry must be
$$  \sum_y L^{--}(x,y) \omega(y) \omega(z) L^{++}(y,z) = 0 \; .  $$ 
To see this, note that set $u$ contributing to $L^{--}(x,y)$ must be contained in $x$
and any set $u$ contributing to $L^{++}(y,z)$ must contain $z$. 
There is no such set and therefore, the answer is $0$.  \\
c) If $z \subset x$ then, the entry
$$  \sum_y L^{--}(x,y) \omega(y) \omega(z) L^{++}(y,z)   $$ 
is a sum of of differences $h(u)^2 - h(v)^2$.  \\
To see this, look at a set $u$ which is a subset of $x$ and 
contains $z$. Show that they appear in pairs. The reason is that the complement of $z$
in $x$ is a complete complex including the empty complex. The even and odd ones
appear with the same cardinality.
\end{proof}

\paragraph{}
This leads immediately to the corollary: 

\begin{thm}
If $h$ takes values in $\{-1,1\}$, then $g$ is the inverse of $L$.
\end{thm}

\paragraph{}
Any ordering of $G$ defines a basis. An other ordering would lead to a conjugated matrices
which would in general no more be lower triangular. For example, if we would order the 
simplices with the largest simplices first, the matrix $L g$ would be upper triangular. 
We see from the proof that if the energy $h$ takes values $\{-a,a\}$, with some real non-zero
$a$, then $L g$ is $a^2$ times the identity matrix, implying the corollary. 

\section{Determinant}

\paragraph{}
The next result deals with determinants. For any real-valued function $h: G \to \mathbb{R}$
define its Fermi characteristic
$$  \prod_{x \in G} h(x) \; . $$ 
It is a multiplicative version of 
the total energy $E[G]=\sum_{x \in G} h(x)$. All the homoclinic or heteroclinic matrices
$L^{++},L^{--},L^{+-},L^{-+}$ have the same determinant. The result is a bit more general now
as it does not require $G$ to be a simplicial complex. It could be a finite topology or finite
Boolean algebra for example.

\begin{thm}[Determinant]
If $G$ is a finite set of sets, then 
${\rm det}(L) = {\rm det}(g) = \prod_{x} h(x)$. 
\end{thm}
\begin{proof} 
The matrix $L^{++}$ has only entries $0$ or $h(x)$ in the last column because 
$L^{++}(x,y)=h(x)$ if $y \subset x$ and $L^{++}(x,y)=0$ else. 
$$ g = \left[ \begin{array}{ccccccc}
    g(1,1) & g(1,2) &     . & . & . & g(1,n)    &  b_1 h(x) \\
    g(2,1) & g(2,2) &     . & . & . & .        &  b_2 h(x) \\ 
      .    & .      &     . & . & . & .      &  b_3 h(x) \\
      .    & .      &     . & . & . & .      & .         \\
      .    & .      &     . & . & . & .      & .         \\
    g(n,1) & .      &     . & . & . & g(n,n) &  b_n h(x) \\ 
    b_1 h(x) & b_2 h(x) &\dots  & \dots & \dots  & b_n h(x) &  h(x)  \\
            \end{array} \right]  \;    $$
where $b_i \in \{0,1\}$. 
Look at all the possible paths in the determinant
which enter the neighborhood of $x$. These interaction paths come in pairs
one which has $x$ as a fixed point and goes $yz$ 
and which which does not and goes $y x z$. These pairs cancel. The only paths left
are the paths which do not interact with $x$. But that means that we have
the situation where $x$ is separated. If $x$ is gone, then any two pair $y,z$  of $S(x)$
are not connected in $g$ as $W^+(y) \cap W^-(y)=0$. 
\end{proof}

\paragraph{}
This implies in the case $h(x) \in \{-1,1\}$ that both matrices $L,g$ are unimodular. If $h(x)$
is constant $1$ \cite{CountingMatrix}, the case where the energy simply counts simplices, 
we additionally know that $L,g$ are in $SL(n,\mathbb{Z})$. 

\begin{coro}[Unimodularity]
If $G$ is a finite set of sets and $h$ takes values in $\{-1,1\}$ then 
${\rm det}(L) = {\rm det}(g) = \prod_{x} h(x)$ takes values in $\{-1,1\}$.
The matrices $L^{++},L^{--},L$ and $g$ are then all unimodular. 
\end{coro}

\paragraph{}
When $h(x)= (-1)^{{\rm dim}(x)}$ we interpreted $\sum_{x \in G} h(x) = {\rm E}[G]$ 
as a Poincar\'e-Hopf result, and interpreted $h(x)$ as an index. 
The formula $\sum_{x \in G} \omega(x) E[S(x)] = {\rm E}[G]$ 
was a dual Poincar\'e-Hopf result used in the original proof of the energy theorem. 

\paragraph{}
Similarly, $\prod_x h(x) = {\rm det}(L)$ can be seen as a multiplicative
Poincar\'e-Hopf result. Actually, one can see the identity
${\rm E}[G]=\sum_x \omega(x) E[W^+(x)]$ as a Poincare-Hopf result in general. 

\section{The eigenvalues}

\paragraph{}
Related it determinants is the next result on eigenvalues. For a symmetric invertible matrix $A$, 
define the Morse index of $A$ as the number of negative eigenvalues of $A$. 
Define also the Morse number of the energy function $h$ as the number of negative eigenvalue 
entries. There is some relation between the function $h$ and the eigenvalues. 
We could formulate it for the matrices $L^{++},L^{--},L^{+-},L^{-+}$, where it is true also 
but do it for the two matrices $L=L^{--}$ and $g=S L^{++} S$: 

\begin{thm}
The number of negative eigenvalues of $L=L^{--}$ is equal to the number of 
negative values of the energy function $h$. The same is true for the matrices $L^{++}$ or $g$. 
\end{thm} 

\begin{proof}
This goes by induction. Let $x$ be the latest cell added to the CW complex. 
Again look at the matrix $g$ and multiply the last entry with $t$:
$$ g = \left[ \begin{array}{ccccccc}
    g(1,1) & g(1,2) &     . & . & . & g(1,n)    &  t b_1 h(x) \\
    g(2,1) & g(2,2) &     . & . & . & .         &  t b_2 h(x) \\ 
      .    & .      &     . & . & . & .         &  t b_3 h(x) \\
      .    & .      &     . & . & . & .         & .         \\
      .    & .      &     . & . & . & .         & .         \\
    g(n,1) & .      &     . & . & . & g(n,n)    &  t b_n h(x) \\ 
t b_1 h(x) & t b_2 h(x) & . & . & . & \dots     & t b_n h(x)   \\
            \end{array} \right]  \;    $$
For $t=0$ we have the determinant $E=\prod_{k=1}^n e(k)$, for $t=1$
we have the determinant $\prod_{k=1}^{n+1} e_k = E h(x)$. Assume there exists a $t_0$
such that the determinant of $g(t_0)=0$. It is linear in $t$.  
A linear function between two positive values is never $0$. 
\end{proof} 

\paragraph{}
This implies:

\begin{coro}
If $h$ is takes only positive values, then both $L,g$ are positive definite.
\end{coro}

\paragraph{}
It also implies that an old result \cite{HearingEulerCharacteristic} that if 
$$ h(x)= \omega(x) = (-1)^{{\rm dim}(x)} = (-1)^{|x|-1} $$
for which the total energy is the Euler characteristic, that the Euler characteristic 
is the number of positive eigenvalues of $g$ minus the number of negative eigenvalues of $g$.

\paragraph{}
Having positive definite integer quadratic forms is always exciting in mathematics. 
It leads to other topics like lattice packings. The matrices $L,g$ are integral quadratic forms
which could serve as a metric in $\mathbb{R}^n$.
Some questions are asked at the end of this document. One can for example ask number
theoretical questions like how large the set of sets is for which the quadratic form
are universal in the context of the Conway-Schneeberger's 15 theorem. 

\section{The potential energy}

\paragraph{}
We now relate the total energy $E[G] = \sum_x h(x)$ of the complex with the
total potential theoretical energy is $\sum_{x,y} g(x,y)$ of $G$. 
The next result tells that these two quantities agree. The theorem requires $G$ 
to be a simplicial complex. It does not work for sets of sets. 

\begin{thm}[Energy theorem for energized complexes]
For any simplicial complex $G$ and any energy function $h$, 
the total energy is $E[G] = \sum_{x,y} g(x,y)$. 
\end{thm}

\paragraph{}
The key is to show that $\omega(x) g(x,x) = \sum_y g(x,y)$ for every $x$
and using the spectral energy result in the next section. But this is
$$ \omega(x) g(x,x) = \sum_y \omega(x) \omega(y) E[ W^+(x) \cap W^+(y)]$$
which is equivalent to 
$$ E[W^+(x) \cap W^+(x)]  = \sum_y \omega(y) E[ W^+(x) \cap W^+(y)] \;  $$
which is the statement ``total energy = spectral energy" in the case $G=W^+(y)$
as the left hand side is the total energy of $W^+(x)$ and $E[ W^+(x) \cap W^+(y)]$
s the diagonal entry $g(y,y)$ if $W^+(x)$ is the total geometry. 

\section{The spectral energy}

\paragraph{}
There is a third energy, the spectral energy which is the sum of the eigenvalues of $S g$. 
Note that in general the eigenvalues of $S g$ are complex because $S g$ is not symmetric. 
But also this spectral energy agrees with the energy. We interpret it as a McKean-Singer
statement which classically is ${\rm str}(e^{-t H}) = \chi(G)$ where $H=(d+d^*)^2$ is the
Hodge Laplacian. 

\begin{thm}[Mc Kean Singer]
$E[G] = {\rm tr}(S g) = {\rm str}(g)$. 
\end{thm}

\begin{proof}
Every simplex $x$ has an energy $h(x)$. The function ${\rm dim}(x)$ is a locally
injective function on the Barycentric graph in which $G$ are the vertices and where two
sets $x,y$ in $G$ are connected, if one is contained in the other. For every simplex $x$, 
let $v(x)$ be a choice of the ${\rm dim}(x)-1$-dimensional simplices 
in $x$. Subtracting the energy from $x$ and
adding it to $v(x)$ does not change the total energy nor does it change the super 
trace. After moving all the energy down to the vertices, the statement is obvious as
$E[ W^+(x) ]$ is now $(-1)^{\rm dim}(x) h(x)$. 
\end{proof} 

\paragraph{}
The fact that the energy is the super trace of $g$ is a discrete 
analogue of the super trace formula $\chi(G) = {\rm str}(e^{-H})$ for the 
Hodge Laplacian $H=(d+d^*)^2$ and Euler characteristic. This is an important step in the proof
as the diagonal entries of $Sg(x,x) = \omega(x) g(x,x)$ can be interpreted as the total 
potential energy $\sum_{y \in G} V_x(y)$ with $V_x(y) = \omega(x) \omega(y) g(x,y)$. 
The proof actually will see this then as a curvature and the theorem as a Gauss-Bonnet
statement for the total energy functional. 

\section{Spectral symmetry}

\paragraph{}
Finally, in the constant energy case $h==1$, there is more symmetry. This has been mentioned
in \cite{CountingMatrix} already. We want to say here more about the proof which is inductive
but still a bit technical even-so we use duality to half the difficulty and use a continuous
deformation argument. 

\begin{thm}[Spectral symmetry]
If the energy function is constant $1$ then $L$ and $g$ are
positive definite. They are inverse to each other and iso-spectral. 
This result holds for arbitrary finite sets of sets $G$. 
\end{thm}

\paragraph{}
It helps to notice that the statement is true even
if $G$ is an arbitrary set of sets and not only true for simplicial complexes.
For sets of sets, there is a duality between stable and unstable parts. 
The axiom of simplicial complexes introduces an asymmetry in
that we require the structure to be invariant under inclusions and not the inclusion
of complements. Now, in the more general frame-work taking a structure $G$ allows to look 
at the dual structure $\hat{G}$. The matrices $L^{++},L^{--}$ for $G$ become then the matrices 
$L^{--},L^{++}$ for the Boolean dual $\hat{G}$. Everything which only concerns 
determinants or spectra is true in general for sets of sets. The spectral symmetry is that 
$g$ and $L$ have the same coefficients in their characteristic polynomial. 

\paragraph{}
As mentioned in \cite{CountingMatrix}, the proof is a deformation argument. But rather than using
a deformation in which the coefficients of the characteristic polynomial are quadratic functions
in $t$ (leading to an ``artillery picture"), we use here a deformation in which the 
coefficients change in a linear manner if one parameter is changed. 

\paragraph{}
The basic idea is to use induction and change the energy of one of the sets $x$ from $h(x)=0$ to 
$h(x)=1$. In the case if $x$ is not contained in any other set, the deformation of $L^{--}$ is easier to 
describe than the deformation of $L^{++}$. For $L^{--}$ only one column changes, while for $L^{++}$
all entries $L(x,y)$ with $y \subset x$ change. If we go to the dual picture $\hat{G}$,
which is the set of complement sets of $G$, 
the matrices $L^{--}$ and $L^{++}$ interchange and the analysis for $L^{--}$ for $\hat{G}$ 
which is analogue to the analysis for $L^{++}$ for $G$ is the same. 

\paragraph{}
We split the deformation into two parts. In a first part, the energies of the energies
``outgoing" from $x$ are throttled. We start with the situation $t=0$, where by induction 
assumption, the coefficients of the characteristic polynomials of $L^{--}$ and $L^{++}$ are 
both palindromic.  This deformation does not preserve the palindromic property of $L^{++}(t)$ 
at $t=1$ but does preserve the palindromic property of $L^{--}(t)$ for $t=1$.

\paragraph{}
Since $L,g$ are both symmetric, they are both diagonalizable using orthogonal matrices.
It follows that $L$ and $g$ are conjugated by orthogonal matrices. By computing eigenbasis
this ``scattering matrix" $Q \in SO(n)$ satisfying $g=Q^{-1} L Q = L^{-1}$. 
This matrix $Q$ is defined uniquely up to permutation and signature change coordinate changes.
If the order with which the set of sets $G$ is build-up is given we still have the choice of
two directions in each eigenvector. There is then a natural ordering of the eigenvalues
$\lambda_k$ of both $L$ and $g$. If $U L U^T = {\rm Diag}(\lambda_1, \cdots, \lambda_n)$ 
and $V g V^T = {\rm Diag}(\lambda_1, \cdots, \lambda_n)$, where the $\lambda_i$ are paired
with the sets $x_i$. This pairing also works if there should be multiple eigenvalues. 
Then $O = U^{-1} V$ conjugates $L$ and $g$.  

\paragraph{}
Example:  \\
1) For the````komma structure"  $G=\{ \{1\},\{1,2\} \}$, which is not a simplicial complex, `
$L=L^{--} = \left[ \begin{array}{cc} 1 & 1 \\ 1 & 2 \\ \end{array} \right]$.
$L^{++}   = \left[ \begin{array}{cc} 2 & 1 \\ 1 & 1 \\ \end{array} \right]$.
The matrix $O=\left[ \begin{array}{cc} 0 & 1 \\ 1 & 0 \\ \end{array} \right]$
conjugates $O L^{--} = L^{++} O$.  
2) For $G=\{ \{1, 2, 3\}, \{1, 2\}, \{1\} \}$, we have
$L^{--} = \left[\begin{array}{ccc} 1 & 1 & 1 \\ 1 & 2 & 2 \\ 1 & 2 & 3 \\ \end{array} \right]$
$L^{++} = \left[\begin{array}{ccc} 3 & 2 & 1 \\ 2 & 2 & 1 \\ 1 & 1 & 1 \\ \end{array} \right]$
$O=\left[ \begin{array}{ccc} 0 & 0 & 1 \\ 0 & 1 & 0 \\ 1 & 0 & 0 \\ \end{array} \right]$.

\paragraph{}
What happens is that the palindromic property in the case $n$ becomes an anti-palindromic
property for $n+1$ and vice versa. The artillery function
$$ A_k(t) = p_k(t)*p_{k+1}(t) - p_{n+2-k}(t)^2  - (-1)^n (p_{k+1}(t) p_{n+2-k}(t) - p_k(t)*p_{n+2-k}(t)) $$
is a quadratic function in $t$. The goal to show that if it is zero for $t=0$
it is zero for $t=1$ is a bit difficult even so we know that the coefficients are sums of minors. 

\begin{figure}[!htpb]
\scalebox{0.5}{\includegraphics{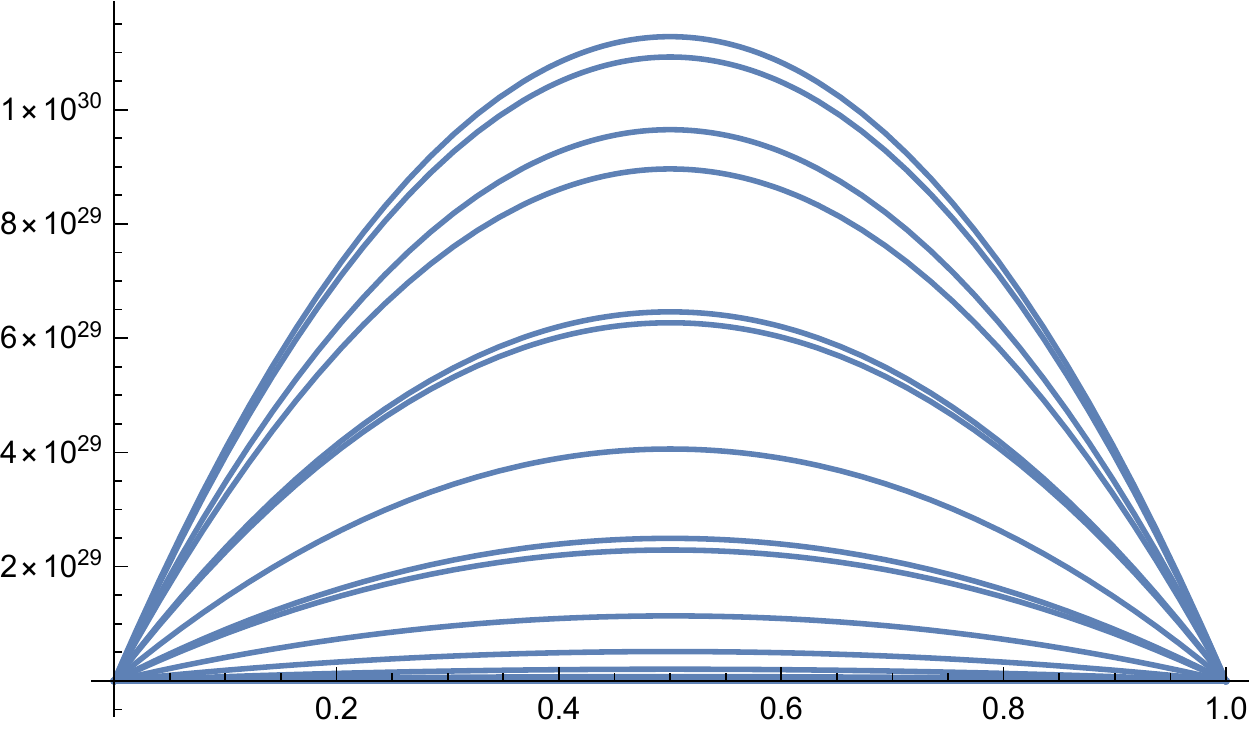}}
\label{artillery}
\caption{
The artillery functions $A_k(t)$ is a quadratic function in $t$. The isospectral property
in the induction step from $n$ to $n+1$ cells means that $A_k(t)=0$ for $t=1$. This means
that the characteristic polynomial remains palindromic.
}
\end{figure}

\paragraph{}
We know already that $L$ and $g$ are positive definite and that $L$ and $g$ 
are inverse of each other. In order to show that $L$ and $g$ are isospectral
we have to show that if $\lambda_k$ is an eigenvalue, then $1/\lambda$ is an 
eigenvalue. This means that the characteristic polynomials 
${\rm det}(L-\lambda)$ and ${\rm det}(g-\lambda)$ agree. Alternatively, this
means that the coefficients 
$$ p(x) = {\rm det}(g-x) = p_0(-x)^n + \cdots +p_k(-x)^{n-k} + \dots + p_n  $$ 
of the characteristic polynomial satisfies the palindromic property 
$$  p_k = \sum_{|P|=k} {\rm det}(g_P) = (-1)^n \sum_{|P|=n-k} {\rm det}(g_P)  = (-1)^n p_{n-k} \;  $$
where ${\rm det}(g_P)$ is a minor, the determinant of a pattern. 
(Proof: $p(\lambda^{-1} = {\rm det}(g-\lambda^{-1}) = {\rm prod}(\lambda_k-\lambda^{-1})$ 
is equal to ${\rm det}(\lambda_k^{-1}-\lambda^{-1}) = (-\lambda^n \prod_k \lambda_k)  {\rm prod} \lambda_k - \lambda 
= (-\lambda^n) {\rm det}(g) {\rm det}(g-\lambda) = (-1)^n \lambda^n p(\lambda)$.)

\paragraph{}
To prove the result, we look at the characteristic polynomials of $L^{++}$ and $L^{--}$ together
and show that they remain palindromic after adding a new cell. 
For the proof we deform the matrices $L^{++}$ and $L^{--}$ with two parameters. One parameter
parameter $T$ scales the column of the cell. The second parameter $H$ is the energy of the cell $x$> 
If the set of sets $G$ is ordered in such a way that the 
last element is not contained in any other set, the parameter $T$ can be 
interpreted as a throttle to release energy from the last element to the others. 
$$ L_{T,R} = \left[ \begin{array}{ccccccc}
    L(1,1) & L(1,2) &     . & . & . & L(1,n)   &  T L(1,n+1) \\
    L(2,1) & L(2,2) &     . & . & . & .        &  T L(2,n+1) \\ 
      .    & .      &     . & . & . & .        &  T L(3,n+1) \\
      .    & .      &     . & . & . & .        & .         \\
      .    & .      &     . & . & . & .        & .         \\
    L(n,1) & .      &     . & . & . & L(n,n)   &  T L(n,n+1) \\ 
    L(n+1,1) & L(n+1,2) &\dots  & \dots & \dots  & \dots & T L(n+1,n+1)   \\
            \end{array} \right]  \;  .  $$

\paragraph{}
In a dual way, we can throttle the energy coming to the last element. As the
last set is not contained in any other this is equivalent to parametrize the 
energy $h(n)=H$ with $H \in [0,1]$. The two situations and deformations are 
analog because of duality. If we go from $G$ to $\hat{G}$, then this
changes the subset ordering to supset ordering and changes the $L^{++}$ to $L^{--}$. 

\paragraph{}
In order to go from a situation with $n$ sets to a situation with $n+1$ sets, we 
first turn on the parameters $T$ and $H$ from $0$ to $1$. 
Each part only preserves the palindromic property of one of the characteristic polynomials
as well as the difference between suitably shifted sequences. \\

\paragraph{}
The coefficients $p_k(T,H)) = q_k(H,T)$ of the characteristic polynomials
of $L^{++}(T,H)$ and $L^{--}(T,H)$ are all multi-linear functions in $H$ and $T$. 
The difference $p_k(T,H)-q_k(T,H)$ is symmetric in $T,H$ and it remains to show
that $T p_k(T,1/T)$ is palindromic for all times. This is a statement which holds
for any energy change of any cell. The $H=1/T$ of the outgoing energy compensates with 
the energy $T$ of the in incoming energies.  

\paragraph{}
The choice of $H=1/T$ is motivated by the fact that with $H=1/T$ the deformed 
matrices remain in $SL(n,\mathbb{Z})$. (We currently believe that this would lead to a
more general result besides the case $h(x)=1$ in which every node has in and out-going
energies which multiply to $1$.)
We get then a quadratic function in $T$, but we only need to show that the first
and second derivatives are palindromic at $T=1$.
Since shooting the artillery leads
for $T=0$ to a palindromic situation and the palindromic property holds along the way, 
we also have a palindromic situation at $T=1$. Now, taking the derivative with respect to $T$
is by the Laplace expansion of the minors a sum of smaller dimensional situations in which 
cells $y_k$ have their energy modified allowing induction on the number $n$ of cells. 

\section{Illustrating the deformation proof}

\paragraph{}
To illustrate the proof, we work with 
$$  G=\{\{1\},\{2\},\{3\},\{4\},\{1,2\},\{2,3\},\{2,4\},\{3,4\} \} \; . $$
We have
$$ L^{--} = \left[ \begin{array}{cccccccc}
                   1 & 0 & 0 & 0 & 1 & 0 & 0 & 0 \\
                   0 & 1 & 0 & 0 & 1 & 1 & 1 & 0 \\
                   0 & 0 & 1 & 0 & 0 & 1 & 0 & 1 \\
                   0 & 0 & 0 & 1 & 0 & 0 & 1 & 1 \\
                   1 & 1 & 0 & 0 & 3 & 1 & 1 & 0 \\
                   0 & 1 & 1 & 0 & 1 & 3 & 1 & 1 \\
                   0 & 1 & 0 & 1 & 1 & 1 & 3 & 1 \\
                   0 & 0 & 1 & 1 & 0 & 1 & 1 & 3 \\
        \end{array} \right]  \; , $$
$$ L^{++} = \left[ \begin{array}{cccccccc}
                   2 & 1 & 0 & 0 & 1 & 0 & 0 & 0 \\
                   1 & 4 & 1 & 1 & 1 & 1 & 1 & 0 \\
                   0 & 1 & 3 & 1 & 0 & 1 & 0 & 1 \\
                   0 & 1 & 1 & 3 & 0 & 0 & 1 & 1 \\
                   1 & 1 & 0 & 0 & 1 & 0 & 0 & 0 \\
                   0 & 1 & 1 & 0 & 0 & 1 & 0 & 0 \\
                   0 & 1 & 0 & 1 & 0 & 0 & 1 & 0 \\
                   0 & 0 & 1 & 1 & 0 & 0 & 0 & 1 \\
                  \end{array} \right]   \; . $$
Both matrices have the palindromic characteristic polynomial
$$ p_L(\lambda) = {\rm det}(L-\lambda) = \lambda^8-16 \lambda^7+95 \lambda^6-268 \lambda^5+380 \lambda^4-268 \lambda^3+95 \lambda^2-16 \lambda+1 \; . $$
The coefficient list of $p_L(-\lambda)$ is $p=(1, 6, 95, 268, 380, 268, 95, 16, 1)$. We have taken $-\lambda$
in order not having to bother with negative signs and palindromic or anti-palindromic situations. 

\paragraph{}
Adding an other cell $\{2,3,4\}$ gives the complex
$$ G=\{\{1\},\{2\},\{3\},\{4\},\{1,2\},\{2,3\},\{2,4\},\{3,4\},\{2,3,4\}\} $$ 
leading to the matrices
$$ L^{--} = \left[ \begin{array}{ccccccccc}
                   1 & 0 & 0 & 0 & 1 & 0 & 0 & 0 & 0 \\
                   0 & 1 & 0 & 0 & 1 & 1 & 1 & 0 & 1 \\
                   0 & 0 & 1 & 0 & 0 & 1 & 0 & 1 & 1 \\
                   0 & 0 & 0 & 1 & 0 & 0 & 1 & 1 & 1 \\
                   1 & 1 & 0 & 0 & 3 & 1 & 1 & 0 & 1 \\
                   0 & 1 & 1 & 0 & 1 & 3 & 1 & 1 & 3 \\
                   0 & 1 & 0 & 1 & 1 & 1 & 3 & 1 & 3 \\
                   0 & 0 & 1 & 1 & 0 & 1 & 1 & 3 & 3 \\
                   0 & 1 & 1 & 1 & 1 & 3 & 3 & 3 & 7 \\
                  \end{array} \right] $$
and 
$$ L^{++} = \left[ \begin{array}{ccccccccc}
                   2 & 1 & 0 & 0 & 1 & 0 & 0 & 0 & 0 \\
                   1 & 5 & 2 & 2 & 1 & 2 & 2 & 1 & 1 \\
                   0 & 2 & 4 & 2 & 0 & 2 & 1 & 2 & 1 \\
                   0 & 2 & 2 & 4 & 0 & 1 & 2 & 2 & 1 \\
                   1 & 1 & 0 & 0 & 1 & 0 & 0 & 0 & 0 \\
                   0 & 2 & 2 & 1 & 0 & 2 & 1 & 1 & 1 \\
                   0 & 2 & 1 & 2 & 0 & 1 & 2 & 1 & 1 \\
                   0 & 1 & 2 & 2 & 0 & 1 & 1 & 2 & 1 \\
                   0 & 1 & 1 & 1 & 0 & 1 & 1 & 1 & 1 \\
                  \end{array} \right] \; . $$
Both matrices have palindromic characteristic polynomials
$$ p_L(\lambda) = {\rm det}(L-\lambda) = -\lambda^9+23 \lambda^8-176 \lambda^7+628 \lambda^6-1167 \lambda^5+
1167 \lambda^4-628 \lambda^3+176 \lambda^2-23 \lambda+1 \; . $$ 
The coefficient list of $p_L(-\lambda)$ is the palindrome 
$q=(1,23,176,628,1167,1167,628,176,23,1)$.

\paragraph{}
In our case, as $\{2,3,4\}$ is not contained in any other set, the matrix $L^{--}$
has not affected the first $9$ rows and columns of $L^{--}$ before the
addition. But the matrix $L^{++}$ has changed every entry $L^{+}(x,y)$ 
if $y \subset x$. The energy of these sets $y$ changed because they are
contained in a common cell of $x$ and $y$. 

\paragraph{}
Now lets change the energy of the cell $x$ and denote it with $H=h(x)$. 
The other energy entries remain $1$. Now we have 
$$ L^{--} =  \left[ \begin{array}{ccccccccc}
          1 & 0 & 0 & 0 & 1 & 0 & 0 & 0 & 0 \\
          0 & 1 & 0 & 0 & 1 & 1 & 1 & 0 & 1 \\
          0 & 0 & 1 & 0 & 0 & 1 & 0 & 1 & 1 \\
          0 & 0 & 0 & 1 & 0 & 0 & 1 & 1 & 1 \\
          1 & 1 & 0 & 0 & 3 & 1 & 1 & 0 & 1 \\
          0 & 1 & 1 & 0 & 1 & 3 & 1 & 1 & 3 \\
          0 & 1 & 0 & 1 & 1 & 1 & 3 & 1 & 3 \\
          0 & 0 & 1 & 1 & 0 & 1 & 1 & 3 & 3 \\
          0 & 1 & 1 & 1 & 1 & 3 & 3 & 3 & H+6 \\
         \end{array} \right] \;  $$
and 
$$ L^{++} =  \left[ \begin{array}{ccccccccc}
                   2 & 1 & 0 & 0 & 1 & 0 & 0 & 0 & 0 \\
                   1 & H+4 & H+1 & H+1 & 1 & H+1 & H+1 & H & H \\
                   0 & H+1 & H+3 & H+1 & 0 & H+1 & H & H+1 & H \\
                   0 & H+1 & H+1 & H+3 & 0 & H & H+1 & H+1 & H \\
                   1 & 1 & 0 & 0 & 1 & 0 & 0 & 0 & 0 \\
                   0 & H+1 & H+1 & H & 0 & H+1 & H & H & H \\
                   0 & H+1 & H & H+1 & 0 & H & H+1 & H & H \\
                   0 & H & H+1 & H+1 & 0 & H & H & H+1 & H \\
                   0 & H & H & H & 0 & H & H & H & H \\
                  \end{array} \right] \; . $$

\paragraph{}
The coefficient list of the characteristic polynomial of $L^{--}$ as a
function of $H$ is 
$p(H)=(H, 7 + 16 H, 81 + 95 H, 360 + 268 H, 787 + 380 H, 899 + 268 H, 
 533 + 95 H, 160 + 16 H, 22 + H, 1)$ and for $L^{++}$ it is
$q(H) = (H, 1 + 22 H, 16 + 160 H, 95 + 533 H, 268 + 899 H, 380 + 787 H, 
 268 + 360 H, 95 + 81 H, 16 + 7 H, 1)$. While $q(0),q(1)=p(1)$ are
palindromic, the characteristic polynomial of $L^{--}(H=0)$ is not but since
$H$ only appears in the corner, the derivative
$p'(H)$ is constant $q(0)$ and palindromic. While the induction step is harder
to handle in the case $L^{--}$, the $L^++$ picture has a better deformation
property in the sense that $q(0)$ and $q(1)$ are both palindromic. The reason
that $p(0)$ is not palindromic is that in the case $H=0$, we still have contributions
of the new cell even so its energy has been turned off. 

\paragraph{}
We therefore have to turn off also the flux of the energy to the new cell $x$ 
We do this with an other parameter $T$. There is now a two parameter family of deformations
$$ L^{--} = \left[
                  \begin{array}{ccccccccc}
                   1 & 0 & 0 & 0 & 1 & 0 & 0 & 0 & 0 \\
                   0 & 1 & 0 & 0 & 1 & 1 & 1 & 0 & T \\
                   0 & 0 & 1 & 0 & 0 & 1 & 0 & 1 & T \\
                   0 & 0 & 0 & 1 & 0 & 0 & 1 & 1 & T \\
                   1 & 1 & 0 & 0 & 3 & 1 & 1 & 0 & T \\
                   0 & 1 & 1 & 0 & 1 & 3 & 1 & 1 & 3 T \\
                   0 & 1 & 0 & 1 & 1 & 1 & 3 & 1 & 3 T \\
                   0 & 0 & 1 & 1 & 0 & 1 & 1 & 3 & 3 T \\
                   0 & 1 & 1 & 1 & 1 & 3 & 3 & 3 & (H+6) T \\
                  \end{array}
                  \right]  \; . $$
It has the characteristic polynomial coefficients 
\begin{tiny}
\begin{eqnarray*}
 p_{T,H} &=& (HT, 1 + 6T + 16HT, 16 + 65T + 95HT, 95 + 265T + 268HT, 268 + 519T + 380HT,  \\
         & & 380 + 519T + 268HT, 268 + 265T + 95HT, 95 + 65T + 16HT, 16 + 6T + HT, 1) \; . 
\end{eqnarray*}
\end{tiny}
When $T=0,H=0$, we have the palindromic property from the case without the cell $x$ but shifted due
to an other eigenvalue $0$. 
$$  T=0,H=0,    p_{00} = (0, 1, 16, 95, 268, 380, 268, 95, 16, 1)   $$
Now, $T=1,H=0$ is not palindromic but $T=0,H=1$ gives the palindromic
characteristic polynomial coefficients:
$$  T=0,H=1,    p_{01} = (0, 1, 16, 95, 268, 380, 268, 95, 16, 1)  \; .  $$
Finally, when $T=1,H=1$, when we allow energy to go into the new cell $x$, 
we have the new situation
$$  T=1,H=1,    p_{11} = (1, 23, 176, 628, 1167, 1167, 628, 176, 23, 1)   \; .  $$

\paragraph{}
To see that the transition from $T=0,H=1$ to $T=1,H=1$ preserves the palindromic property we look at
$$ L^{++} = \left[
                  \begin{array}{ccccccccc}
                   2 & 1 & 0 & 0 & 1 & 0 & 0 & 0 & 0 \\
                   1 & H+4 & H+1 & H+1 & 1 & H+1 & H+1 & H & H T \\
                   0 & H+1 & H+3 & H+1 & 0 & H+1 & H & H+1 & H T \\
                   0 & H+1 & H+1 & H+3 & 0 & H & H+1 & H+1 & H T \\
                   1 & 1 & 0 & 0 & 1 & 0 & 0 & 0 & 0 \\
                   0 & H+1 & H+1 & H & 0 & H+1 & H & H & H T \\
                   0 & H+1 & H & H+1 & 0 & H & H+1 & H & H T \\
                   0 & H & H+1 & H+1 & 0 & H & H & H+1 & H T \\
                   0 & H & H & H & 0 & H & H & H & H T \\
                  \end{array}
                  \right] $$
which has the coefficients
\begin{tiny}
\begin{eqnarray*}
 q_{T,H} &=& (HT, 1 + 6H + 16*HT, 16 + 65H + 95HT, 95 + 265H + 268HT,  \\
         & & 268 + 519H + 380HT, 380 + 519H + 268HT, 268 + 265H + 95HT, 95 + 65H + 16HT, 16 + 6H + HT, 1) 
\end{eqnarray*}
\end{tiny}
which satisfies $q_{T,H}=p_{H,T}$. The difference $q_{T,H}-p_{T,H}$ is palindromic
\begin{tiny}
\begin{eqnarray*}
 q_{T,H}-p_{T,H} &=&  (0, -6H + 6T, -65H + 65T, -265H + 265T, -519H + 519T, -519H + 519T,  \\
                 & & -265H + 265T, -65H + 65T, -6H + 6T, 0) \; .
\end{eqnarray*}
\end{tiny}
We have 
\begin{tiny}
\begin{eqnarray*}
 q_{T,T} &=& p_{T,T} = (T^2,16 T^2+6 T+1,95 T^2+65 T+16,268 T^2+265 T+95,380 T^2+519 \\
         & & T+268,268 T^2+519 T+380,95 T^2+265 T+268,16 T^2+65 T+95,T^2+6 T+16,1 )  \; . 
\end{eqnarray*}
\end{tiny}
The palindromic property holds for $T=0$ to $T=1$. If we look at 
\begin{tiny}
\begin{eqnarray*}
 T q_{T,1/T} &=& T p_{T,1/T} = 
   (T,6 T^2+17 T,65 T^2+111 T,265 T^2+363 T,519 T^2+648 T,519 T^2+648 T,265 \\
   & & T^2+363 T,65 T^2+111 T,6 T^2+17 T,T ) \; .
\end{eqnarray*}
\end{tiny}
The palindromic property for $q_{1,1}$ follows from the statement that 
$T q_{T,1/T}$ is palindromic for all $T$. The statement for $p_{1,1}$ then 
follows from duality. Let us look at this next: 

\paragraph{}
When going from $G$ to the dual 
$$ \hat{G}=\{\{1\},\{1,2\},\{1,3\},\{1,4\},\{3,4\},\{1,2,3\},\{1,2,4\},\{1,3,4\},\{2,3,4\}\} $$
we reverse $L^{++}$ and $L^{--}$. 
If we would have added the largest cell, we 
get just a reverse of $L^{--}$ and $L^{++}$. Without resorting $\hat{G}$, we have
$$ L^{--} = \left[
                  \begin{array}{ccccccccc}
                   2 & 1 & 0 & 0 & 1 & 0 & 0 & 0 & 0 \\
                   1 & H+4 & H+1 & H+1 & 1 & H+1 & H+1 & H & H T \\
                   0 & H+1 & H+3 & H+1 & 0 & H+1 & H & H+1 & H T \\
                   0 & H+1 & H+1 & H+3 & 0 & H & H+1 & H+1 & H T \\
                   1 & 1 & 0 & 0 & 1 & 0 & 0 & 0 & 0 \\
                   0 & H+1 & H+1 & H & 0 & H+1 & H & H & H T \\
                   0 & H+1 & H & H+1 & 0 & H & H+1 & H & H T \\
                   0 & H & H+1 & H+1 & 0 & H & H & H+1 & H T \\
                   0 & H & H & H & 0 & H & H & H & H T \\
                  \end{array}
                  \right] $$
and 
$$ L^{++} =  \left[
                  \begin{array}{ccccccccc}
                   1 & 0 & 0 & 0 & 1 & 0 & 0 & 0 & 0 \\
                   0 & 1 & 0 & 0 & 1 & 1 & 1 & 0 & T \\
                   0 & 0 & 1 & 0 & 0 & 1 & 0 & 1 & T \\
                   0 & 0 & 0 & 1 & 0 & 0 & 1 & 1 & T \\
                   1 & 1 & 0 & 0 & 3 & 1 & 1 & 0 & T \\
                   0 & 1 & 1 & 0 & 1 & 3 & 1 & 1 & 3 T \\
                   0 & 1 & 0 & 1 & 1 & 1 & 3 & 1 & 3 T \\
                   0 & 0 & 1 & 1 & 0 & 1 & 1 & 3 & 3 T \\
                   0 & 1 & 1 & 1 & 1 & 3 & 3 & 3 & (H+6) T \\
                  \end{array}
                  \right] $$
So, if we make the induction assumption that for all sets of sets $G$ with $n$
elements, the palindromic property for $L^{--},L^{++}$ implies the palindromic property for 
$L^{--}$ in the case $n+1$, then we also have that in the dual picture
the palindromic property for $L^{++}$. We can therefore focus on one case. 

\paragraph{}
So, the proof reduced to show that $T p(T,1/T)=T q(1/T,T)$ is a palindromic sequence
of quadratic functions in $T$ for all $T$. This is shown if $T p(T,1/T)$

\section{Zeta function} 

\paragraph{}
As $L,g$ are now positive definite quadratic
forms which are isospectral one can define a spectral zeta function
$$   \zeta(s) = \sum_{k=1}^n \lambda_k^{-s}  \;. $$
One can also look at the Ihara zeta function $\zeta_I(s) = 1/\det(1-s L)$.

\begin{coro}
Both the spectral zeta function as well as the Ihara zeta function 
satisfy in the constant energy case $h=1$ the functional equation $\zeta(a+ib)=\zeta(-a+ib)$
respectively $\zeta_I(-a+ib) = \zeta_I(a+ib)$. 
\end{coro}
\begin{proof}
Since the eigenvalues $\lambda_k$ are real, we have $\zeta(a+ib) = \zeta(a-ib)$
and $\zeta_I(a+ib)=\zeta_I(a-ib)$. 
\end{proof} 

\paragraph{}
Of particular interest is the case when $h(x)=1$ for all $x$ as then, 
we not only have the spectral symmetry, but additional know that
the matrices are in $SL(n,\mathbb{R})$. They define positive definite quadratic
quadratic forms, so that: 

\begin{coro}
For every set of sets $G$, the matrices $L^{++}$ and $L^{--}$ define unimodular 
integral lattices defined by a positive definite quadratic form of determinant $1$.
\end{coro}

\section{Isospectral multigraphs}

\paragraph{}
An other consequence of the spectral symmetry concerns multigraphs as illustrated
in Figure~(\ref{multigraphs1}), Figure~(\ref{multigraphs2}),
Figure~(\ref{multigraphs3}) and Figure~(\ref{multigraphs4}). 
If $G$ is a set of sets, we can define the multi graphs $\Gamma^{--}$ and $\Gamma^{++}$
in which the sets are the nodes and where two sets $x,y$ are connected using 
exactly $L^{--}(x,y)$ or $L^{--}(x,y)$ connections from $x$ to $y$. 
As the adjacency matrices $L^{++}$ and $L^{--}$ are isospectral, the graphs are isospectral
and especially, the number of closed paths of length $k$ or length $n-k$ are the same. 

\begin{coro}
For any set of sets $G$, the multigraphs defined by the adjacency matrices 
$L^{--}$ and $L^{++}$ are isospectral.
\end{coro}

\paragraph{}
The picture for the random walk on the graph $\Gamma^{--}$ is that we can hop from set to set but
take on the way from $x$ to $y$ a side step to a set in the intersection of $x$ and $y$.
This is encoded in the matrix entry $L^{--}(x,y)$ which count the total energy of the sets
contained in $x \cap y$. The dual picture is that we can jump from $x$ to $y$ while side-stepping
onto a set containing both $x$ and $y$. These two situations can be looked at also in the dual case
where $\hat{G}$ is the set of complements of sets in $G$. 
It is a primitive model in which $x,y$ are entities and the sets in $x \cap y$ or $x \cup y$ 
model common shared interactions or (in the dual case) common bridge operations. 

\paragraph{}
Various zeta functions have been defined for graphs. One can look at the sums $\sum_k \lambda_k^{-s}$
where $\lambda_k$ are the eigenvalues of a Laplacian defined on the graph. Example of ``Laplacians"
are Kirchhoff Laplacians, Hodge Laplacian, connection Laplacians. The adjacency matrix works if it
is positive definite. One can also look at the Ihara zeta function $\zeta(s) = 1/{\rm det}(1-sL)$. 

\begin{coro}
The graphs $\Gamma^{++}$ and $\Gamma^{--}$ have the same spectral and the same Ihara zeta function. 
\end{coro}

\paragraph{}
The coefficients $p_k$ of the characteristic polynomial of $L$ or $g$ have a geometric interpretation 
as the number of closed prime paths of length $k$, meaning that these are simple paths. In contrary, 
${\rm tr}(L^k)$ super counts the number of all closed paths of length $k$. The spectral symmetry result
implies that both for $\Gamma=\Gamma^{++}$ and $\Gamma=\Gamma^{--}$ we have:

\begin{coro}
The super count $p_k$ of all closed simple paths of length $k$ in $\Gamma$ either has 
the symmetry $p_k=p_{n-k}$ or $p_k=-p_{n-k}$.
\end{coro}

\begin{figure}[!htpb]
\scalebox{0.5}{\includegraphics{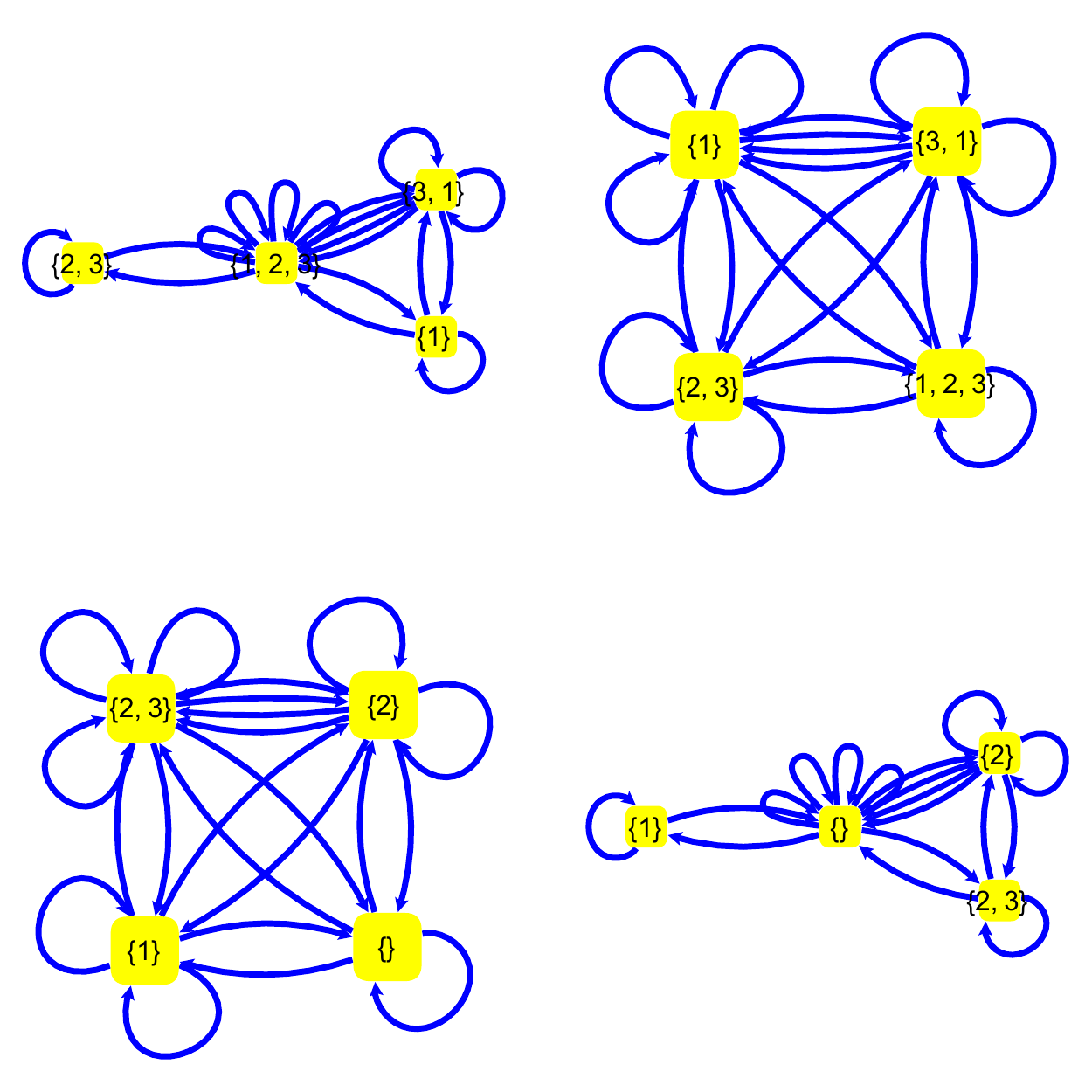}}
\label{multigraphs1}
\caption{
The graphs $\Gamma^{--}$ and $\Gamma^{++}$ for the set of sets $G=\{\{1\},\{2,3\},\{3,1\},\{1,2,3\}\}$
have the sets as nodes and $L^{--}(x,y)$ rsp $L^{++}(x,y)$ connections from $x$ to $y$.
Below we see the dual set $\hat{G}= \{ 2,3\},\{1\},\{2\},\emptyset \}$ and its corresponding
graphs $\hat{\Gamma}^{--}$ and $\hat{\Gamma}^{++}$.
}
\end{figure}

\begin{figure}[!htpb]
\scalebox{0.5}{\includegraphics{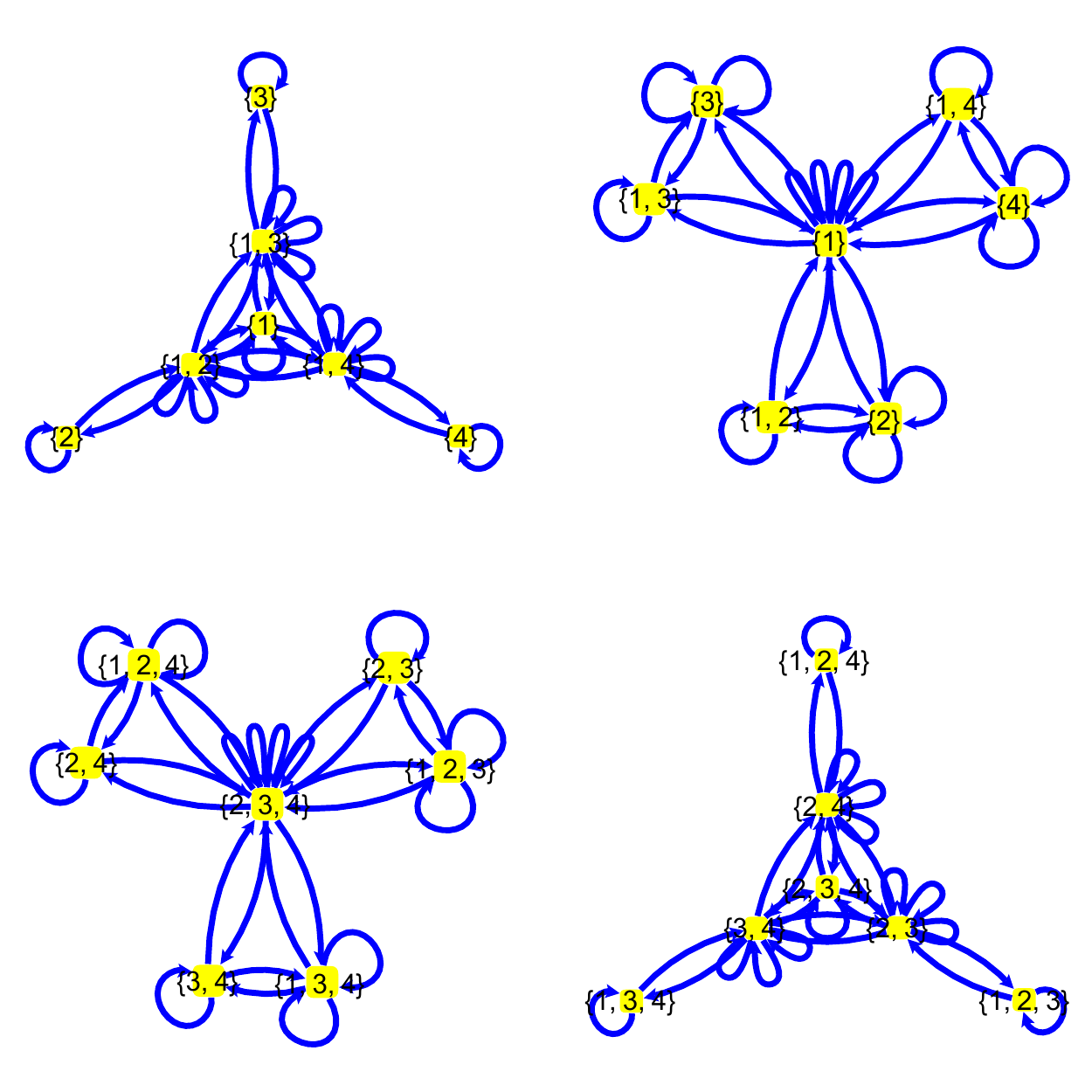}}
\label{multigraphs2}
\caption{
The graphs $\Gamma^{--}$ and $\Gamma^{++}$ for $G=\{\{1\},\{1,2\},\{1,3\},\{1,4\}\}$
and the dual structure. For any set of sets we get isospectral multi-graphs.
}
\end{figure}

\begin{figure}[!htpb]
\scalebox{0.5}{\includegraphics{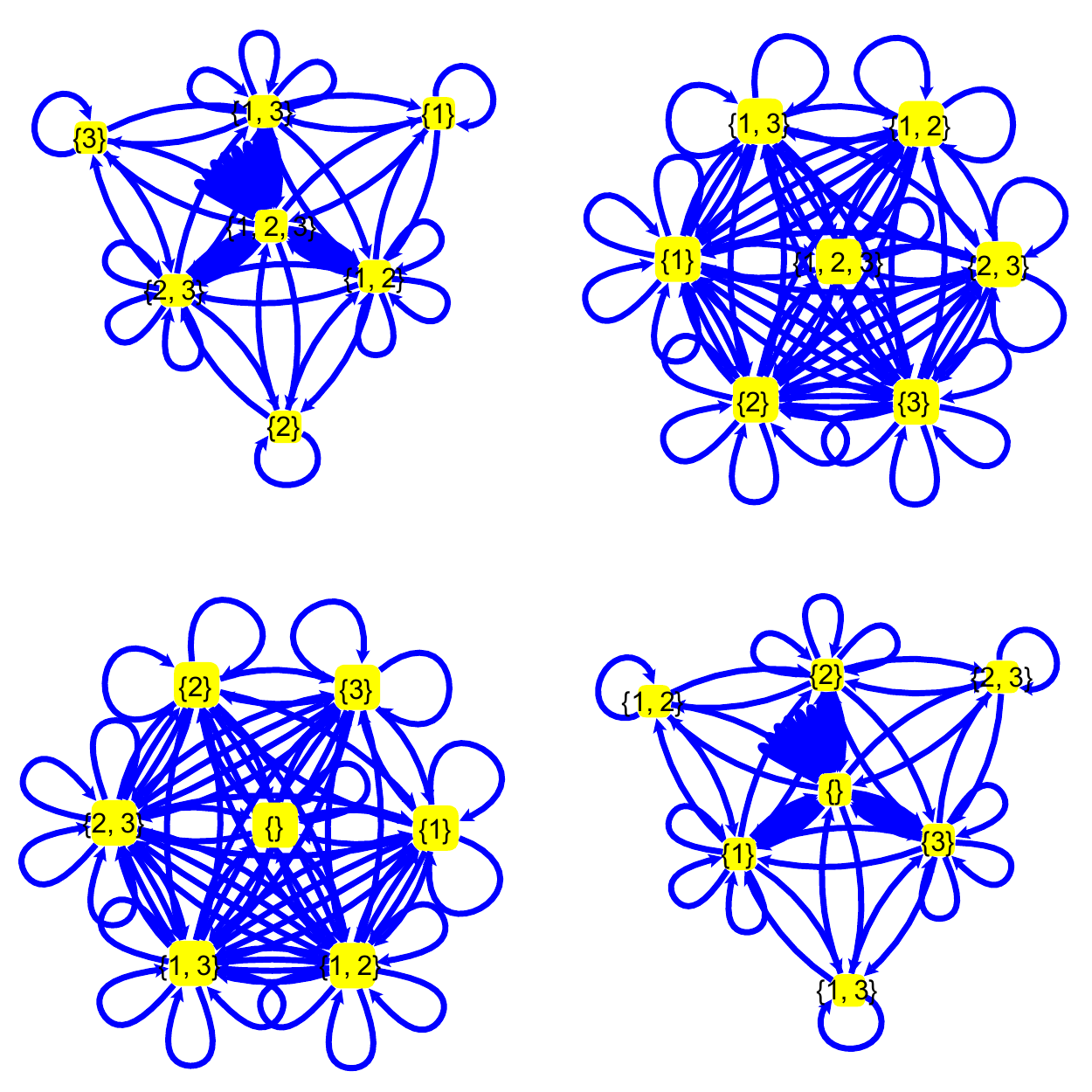}}
\label{multigraphs3}
\caption{
The graphs $\Gamma^{--}$ and $\Gamma^{++}$ for the complete complex $G$
of all non-empty subsets of $\{1,2,3\}$. 
}
\end{figure}

\begin{figure}[!htpb]
\scalebox{0.5}{\includegraphics{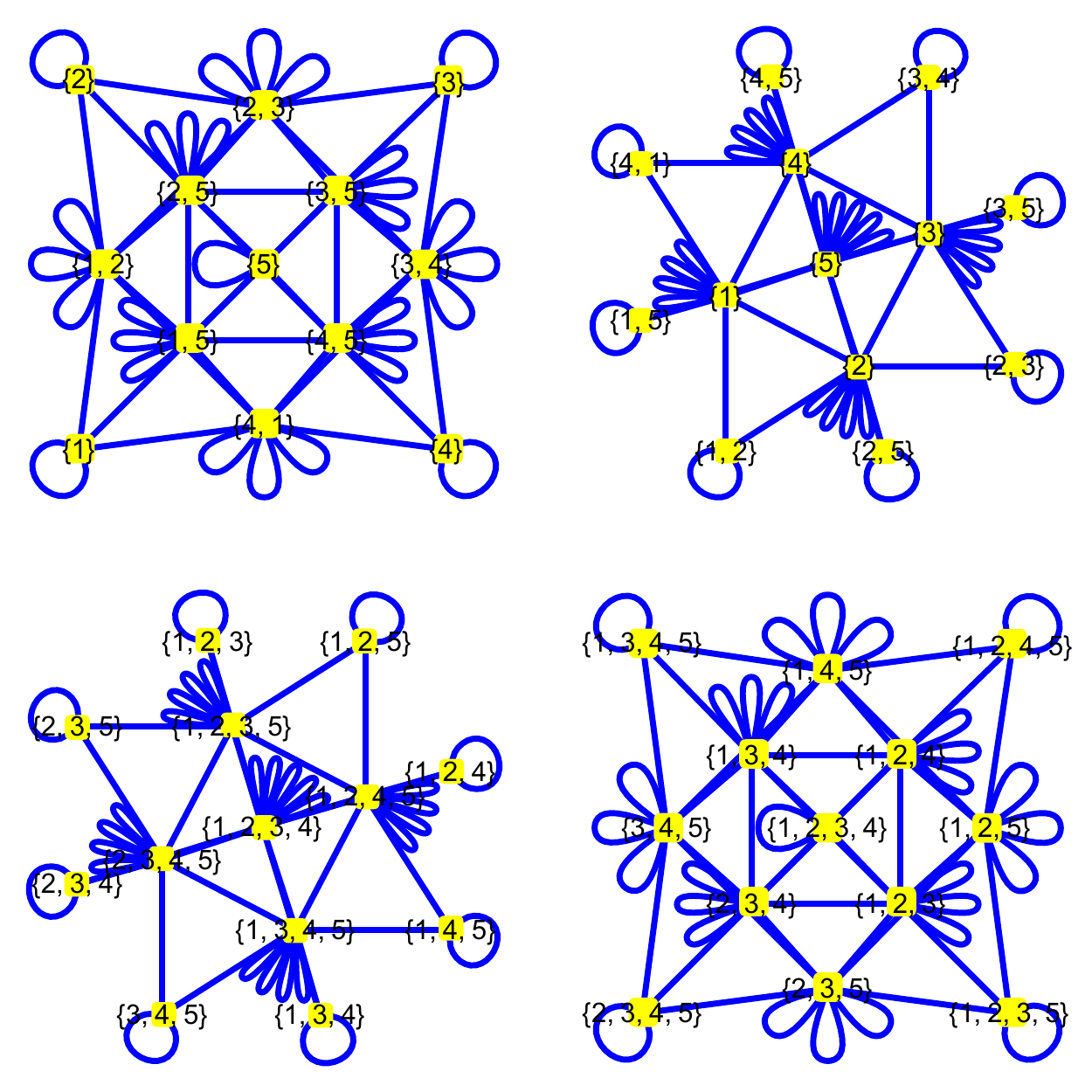}}
\label{multigraphs4}
\caption{
The graphs $\Gamma^{--}$ and $\Gamma^{++}$ for  wheel graph.
}
\end{figure}

\paragraph{}
Finally, we should mention that the construction allows to construct many periodic
and so also almost periodic or even random graphs which are isospectral. 
Just start with a sequence $\{0,1\}$ and encode this into a graph by attaching little
dangles at the points where we have a $1$ and none else. 

\begin{figure}[!htpb]
\scalebox{0.5}{\includegraphics{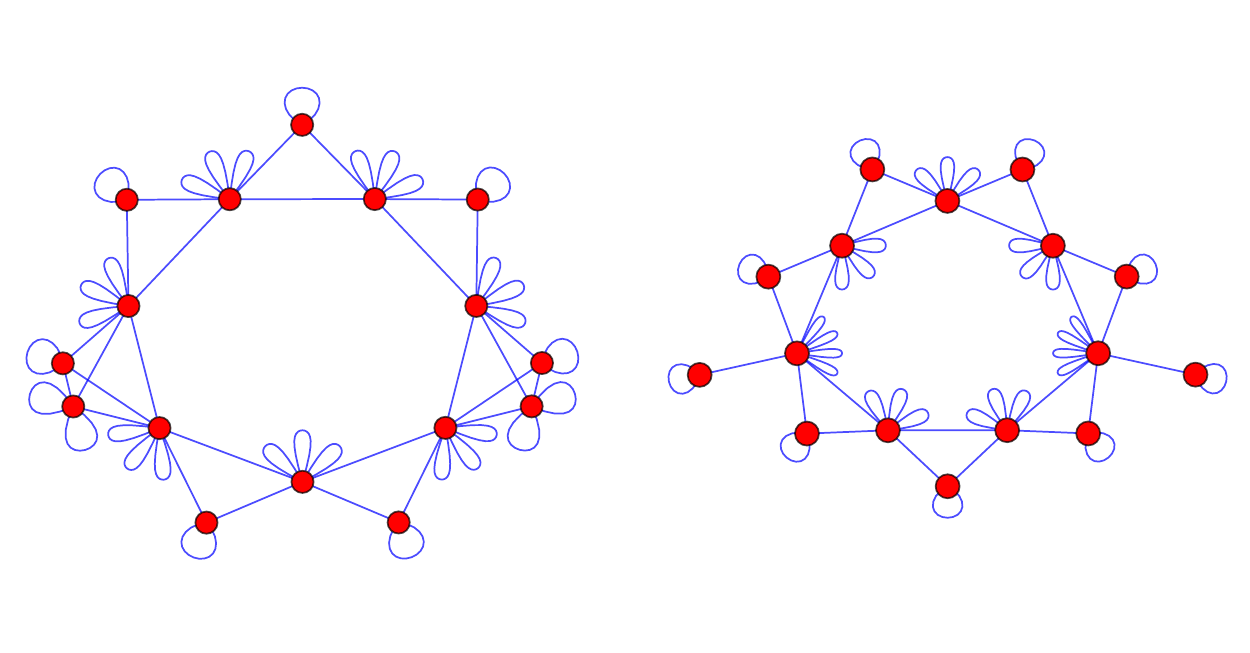}}
\label{multigraphs}
\caption{
Constructing periodic and almost periodic matrices which are isospectral. 
Here is a small periodic isospectral pair of graphs
}
\end{figure}

\begin{figure}[!htpb]
\scalebox{0.6}{\includegraphics{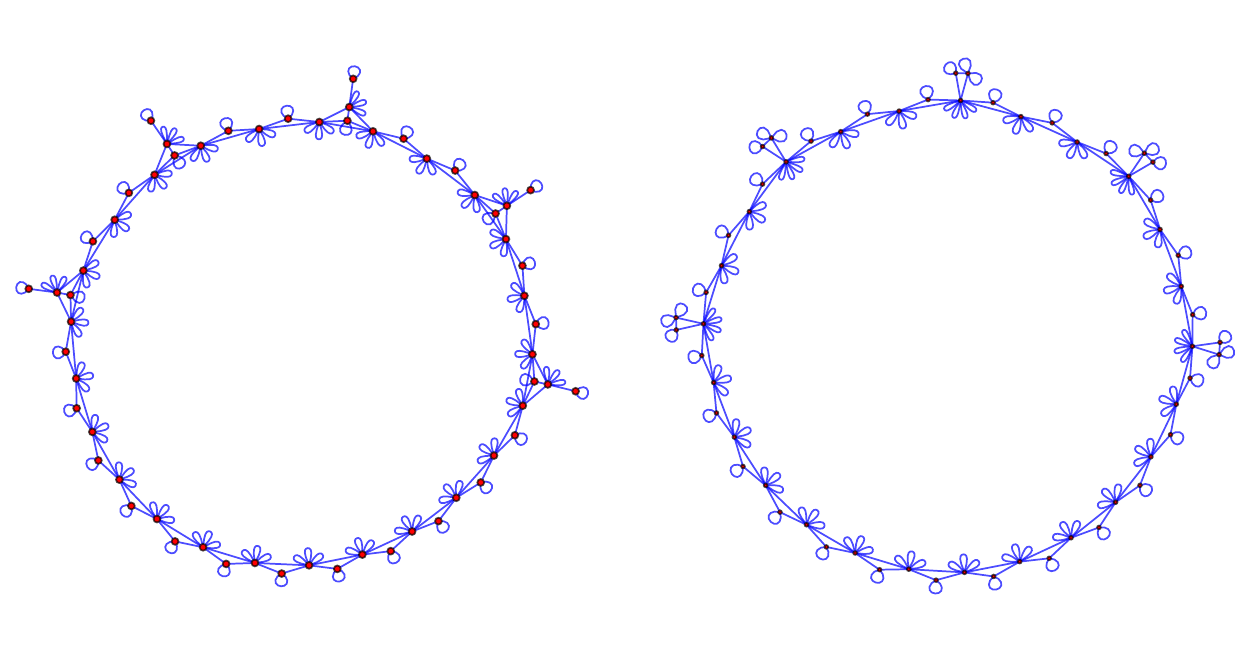}}
\label{multigraphs}
\caption{
A larger isospectral pair of multi-graphs
}
\end{figure}

\begin{figure}[!htpb]
\scalebox{0.3}{\includegraphics{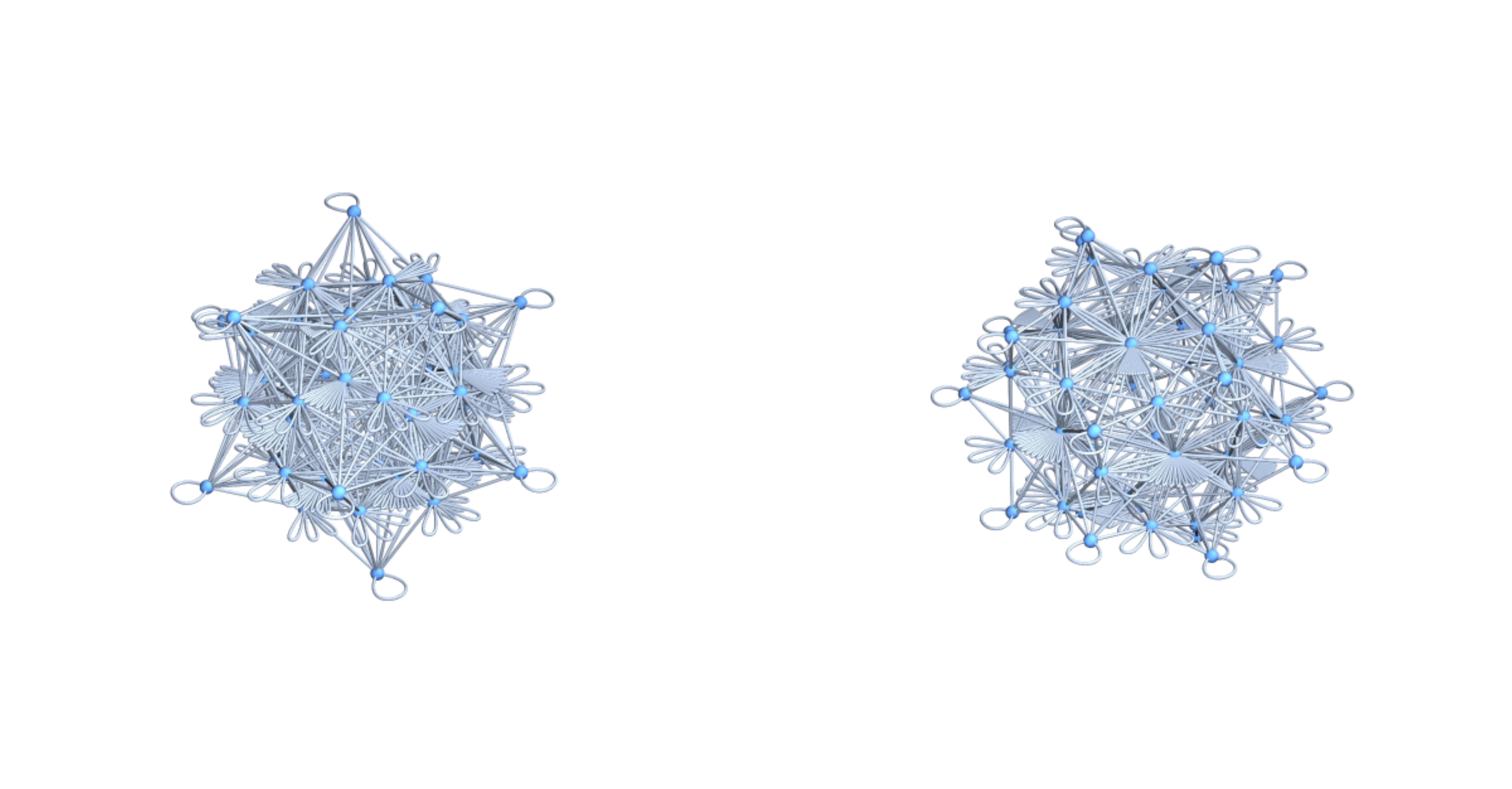}}
\label{multigraphs}
\caption{
The multigraph and dual multigraph for the Whitney
complex of the icosahedron graph. 
}
\end{figure}

\begin{figure}[!htpb]
\scalebox{0.3}{\includegraphics{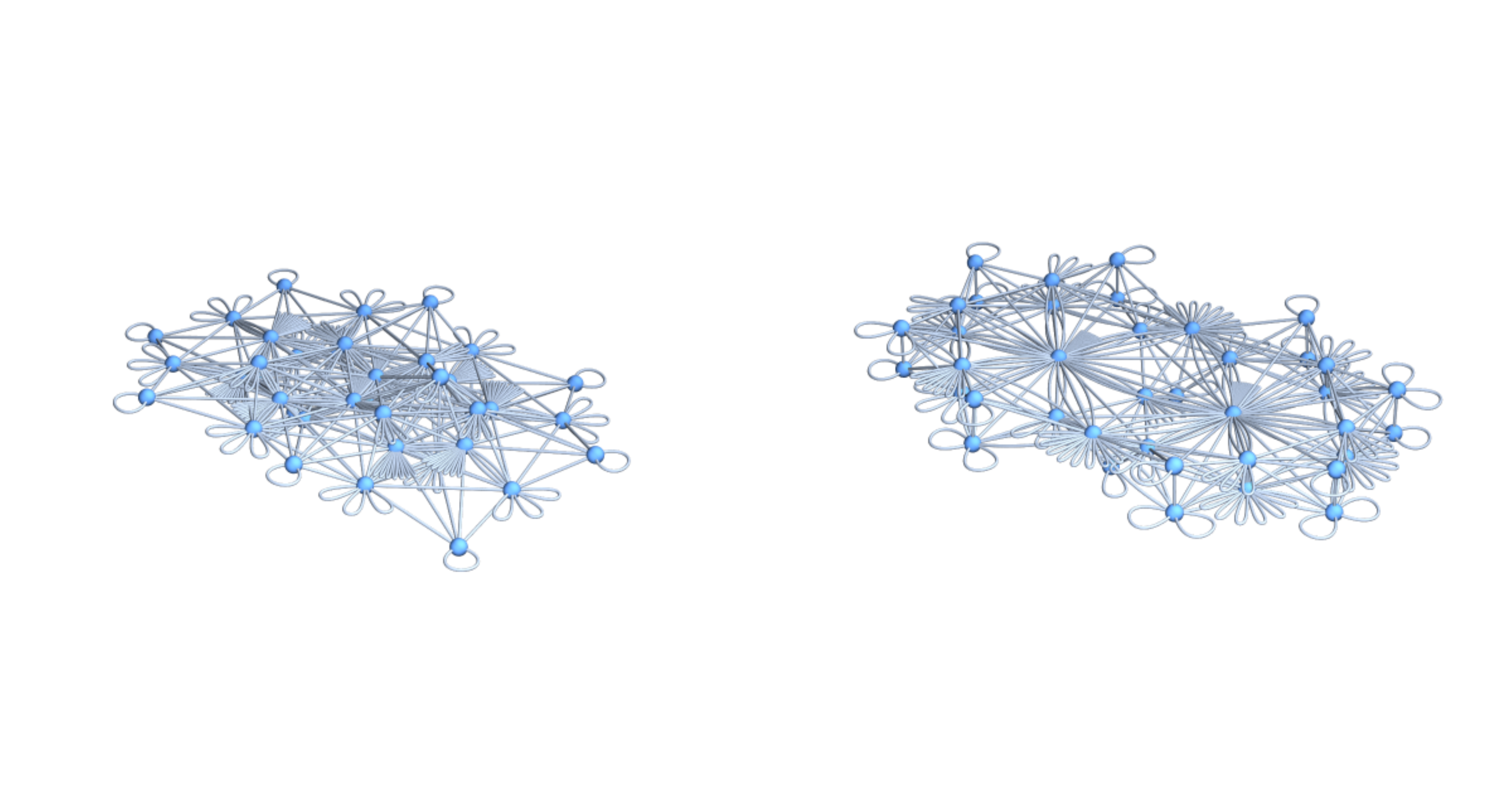}}
\label{multigraphs}
\caption{
The multigraph and dual multigraph for the Whitney
complex of a simple lattice region. We can take the 
van Hove limit and get pairs of periodic infinite matrices which are
isospectral and have continuuous spectrum. 
}
\end{figure}

\begin{figure}[!htpb]
\scalebox{0.3}{\includegraphics{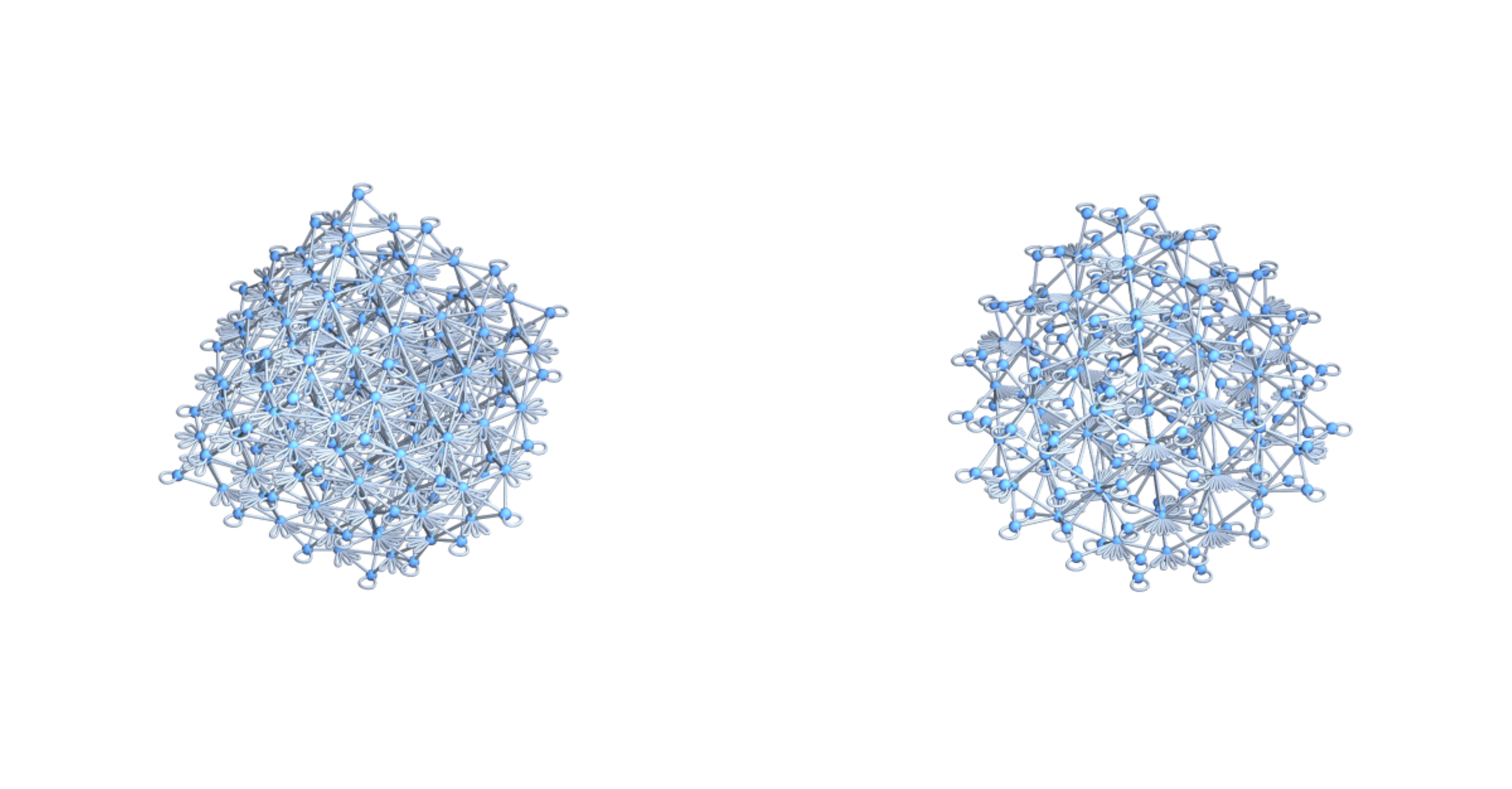}}
\label{multigraphs}
\caption{
The same in any dimension. Any simplicial complex $G$ defines
two isospectral multi-graphs $\Gamma^+,\Gamma^-$. If $G$ is a
Barycentric refinement, then $G,\Gamma^+,\Gamma^-$ are homotopic. 
}
\end{figure}

\section{Examples}

\paragraph{}
{\bf 1)} Let $G=\{ \{1\},\{2\},\{1,2\} \}$ and let $x,y,z$ be the energy values
$h(\{1\})=x,h(\{2\})=y,h(\{z\})=z$. Then 
$$ L = L^{--} = \left[ \begin{array}{ccc}
 x & 0 & x \\
 0 & y & y \\
 x & y & x+y+z \\
\end{array} \right],  
   L^{++} = \left[ \begin{array}{ccc}
 x+z & z & z \\
 z & y+z & z \\
 z & z & z \\
\end{array} \right] $$
with determinants ${\rm det}(L)={\rm det}(g) = xyz$ and 
$$ L g =  \left[ \begin{array}{ccc} 
 x^2 & 0 & 0 \\
 0 & y^2 & 0 \\
 x^2-z^2 & y^2-z^2 & z^2 \\
\end{array} \right]  \; . $$
This shows that if $x,y,z$ take values in $\{-1,1\}$, then $L,SgS$ are
inverses of each other. An other special case is if $x=y=z=a$ is constant, 
then the $L$ and $g$ have the same characteristic polynomial 
$p_L(t)=p_g(t)=a^3-5 a^2 t+5 a t^2-t^3$ and are therefore isospectral. 

\paragraph{}
{\bf 2)} For $G = \{\{1\},\{2\},\{3\},\{1,2\},\{1,3\},\{2,3\},\{1,2,3\}\}$, 
and energy values $\{a, b, c, d, x, y, z\}$, we have
$$ L=L^{--} = \left[
\begin{array}{ccccccc}
 a & 0 & 0 & a & a & 0 & a \\
 0 & b & 0 & b & 0 & b & b \\
 0 & 0 & c & 0 & c & c & c \\
 a & b & 0 & a+b+d & a & b & a+b+d \\
 a & 0 & c & a & a+c+x & c & a+c+x \\
 0 & b & c & b & c & b+c+y & b+c+y \\
 a & b & c & a+b+d & a+c+x & b+c+y & a+b+c+d+x+y+z \\
\end{array} \right] $$
and 
$$
L^{++} = \left[ \begin{array}{ccccccc}
 a+d+x+z & d+z & x+z & d+z & x+z & z & z \\
 d+z & b+d+y+z & y+z & d+z & z & y+z & z \\
 x+z & y+z & c+x+y+z & z & x+z & y+z & z \\
 d+z & d+z & z & d+z & z & z & z \\
 x+z & z & x+z & z & x+z & z & z \\
 z & y+z & y+z & z & z & y+z & z \\
 z & z & z & z & z & z & z \\
\end{array} \right] $$
which both have the determinant $a b c d x y z$. 
The sum of the matrix elements of 
$$ g = \left[ \begin{array}{ccccccc}
 a+d+x+z & d+z & x+z & -d-z & -x-z & -z & z \\
 d+z & b+d+y+z & y+z & -d-z & -z & -y-z & z \\
 x+z & y+z & c+x+y+z & -z & -x-z & -y-z & z \\
 -d-z & -d-z & -z & d+z & z & z & -z \\
 -x-z & -z & -x-z & z & x+z & z & -z \\
 -z & -y-z & -y-z & z & z & y+z & -z \\
 z & z & z & -z & -z & -z & z \\
\end{array} \right] $$
is the total energy $a+b+c+d+x+y+z$. The matrix
$$ 
\scalemath{0.8}{
L  g  = \left[ \begin{array}{ccccccc}
 a^2 & 0 & 0 & 0 & 0 & 0 & 0 \\
 0 & b^2 & 0 & 0 & 0 & 0 & 0 \\
 0 & 0 & c^2 & 0 & 0 & 0 & 0 \\
 a^2-d^2 & b^2-d^2 & 0 & d^2 & 0 & 0 & 0 \\
 a^2-x^2 & 0 & c^2-x^2 & 0 & x^2 & 0 & 0 \\
 0 & b^2-y^2 & c^2-y^2 & 0 & 0 & y^2 & 0 \\
 a^2-d^2-x^2+z^2 & b^2-d^2-y^2+z^2 & c^2-x^2-y^2+z^2 & d^2-z^2 & x^2-z^2 & y^2-z^2 & z^2 \\
\end{array} \right] 
}
$$
shows that if $x \in \{-1,1\}$, then $L$ and $S g S$ are inverse matrices.
If all energy values are equal to $a$, then both $L$ and $S g S$ have the 
characteristic polynomial 
$a^7-19 a^6 t+102 a^5 t^2-228 a^4 t^3+228 a^3 t^4-102 a^2 t^5+19 a t^6-t^7$. 

\begin{figure}[!htpb]
\scalebox{0.2}{\includegraphics{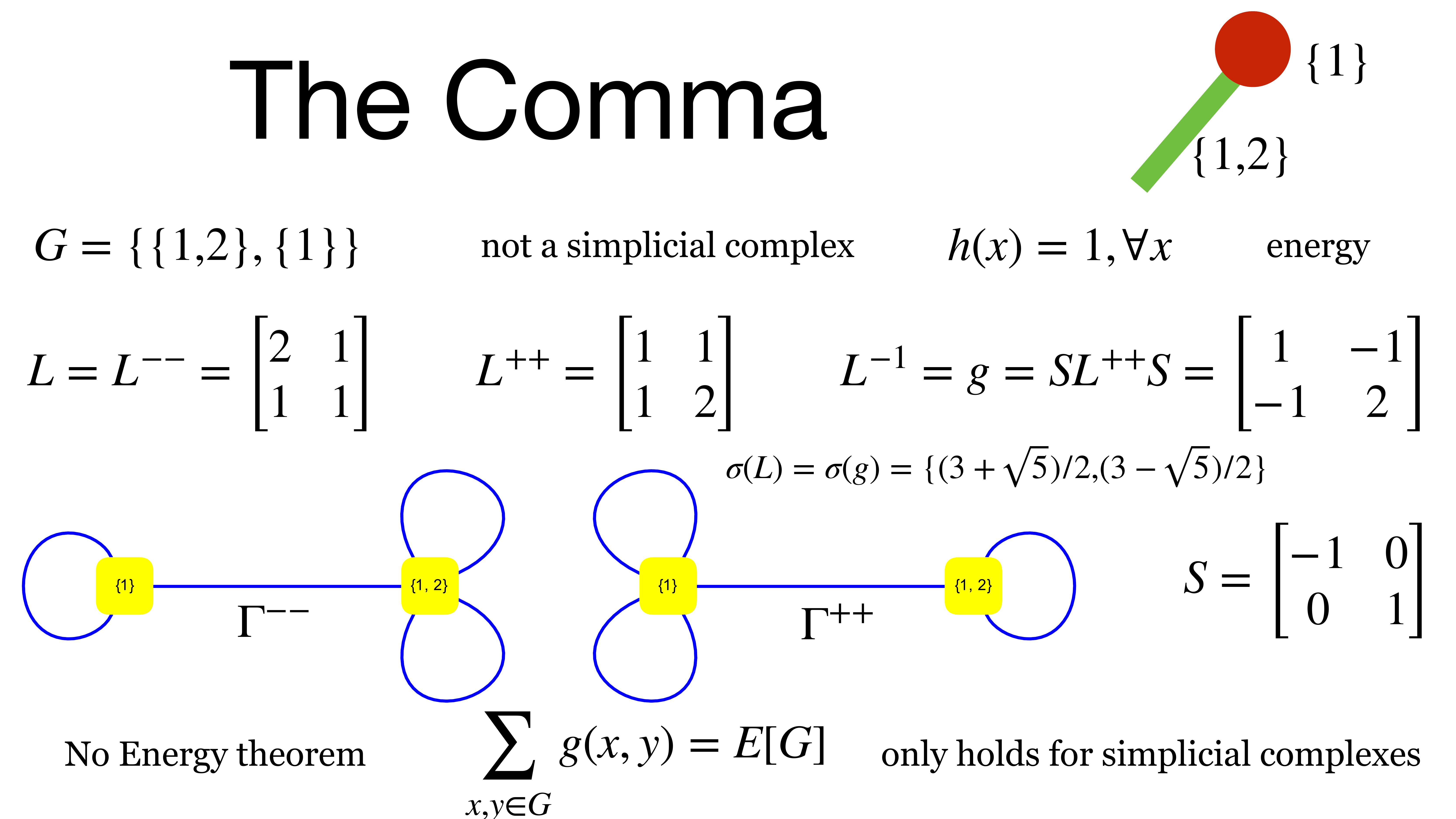}}
\label{energized1}
\caption{
A figure illustrating the story in one of the simplest cases.
$G$ is not a simplicial complex here. In the case of a constant energy $1$,
the matrices $L^{--}$ and $L^{++}$ define isospectral multi-graphs.
The matrix $g=S L^{++} S$ is the inverse of $L=L^{--}$. The energy theorem
does not hold here.
}
\end{figure}

\begin{figure}[!htpb]
\scalebox{0.2}{\includegraphics{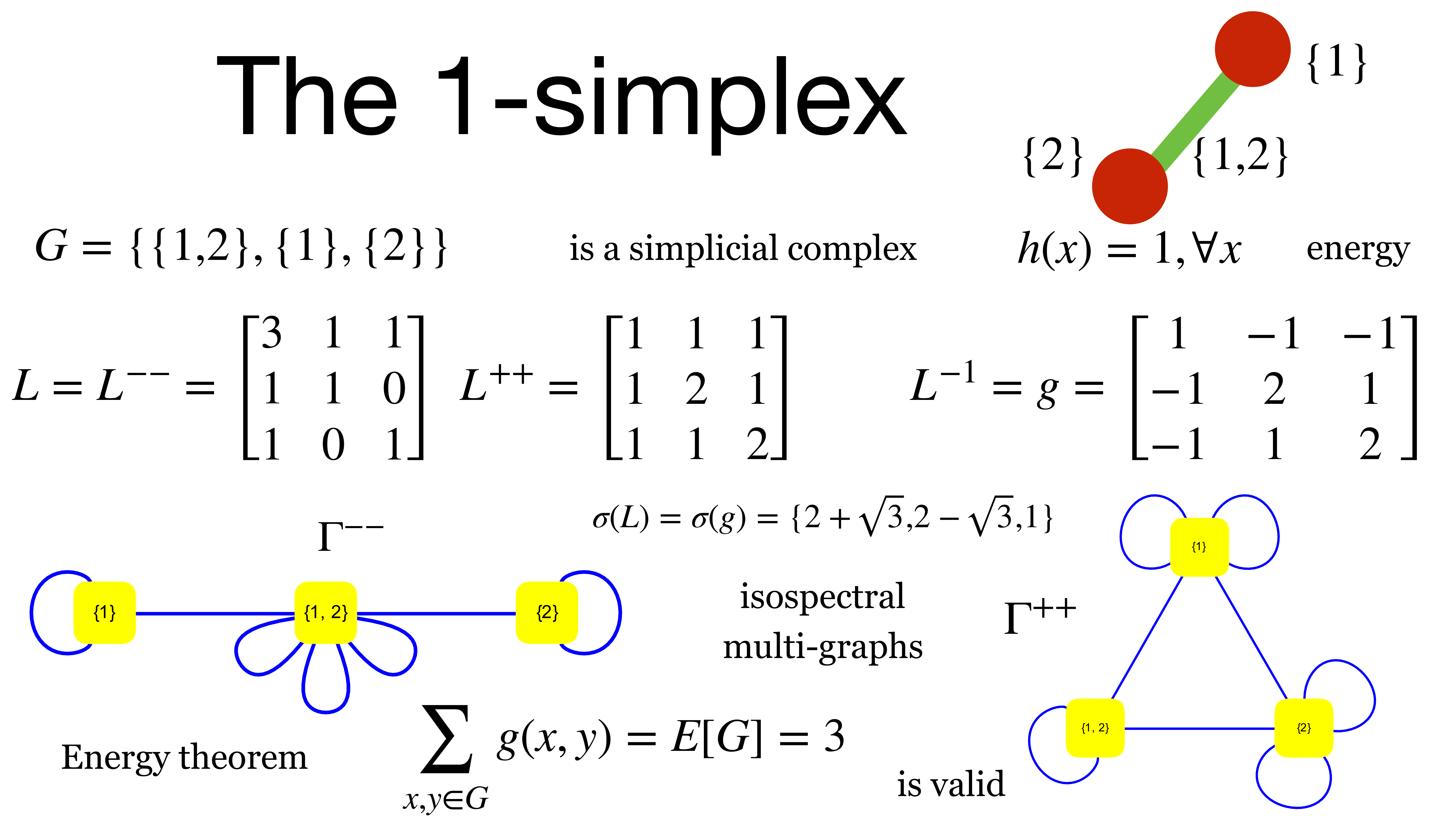}}
\label{energized2}
\caption{
In this case $G$ is a simplicial complex with three sets.
In the case of a constant energy $1$,
the matrices $L^{--}$ and $L^{++}$ define isospectral multi-graphs.
The matrix $g=S L^{++} S$ is the inverse of $L=L^{--}$. The energy theorem
does hold here. The total energy is $3$, the number of sets.
With the energy $\omega(x)=(-1)^{{\rm dim}(x)}$ the
total energy would have been $\chi(G)=1$, the Euler characteristic of $G$.
}
\end{figure}

\section{The parameter case}

\paragraph{}
Define ${\rm dim}(x)=|x|-1$ and $\omega(x)=(-1)^{{\rm dim}(x)}$ and the connection matrix
$$  L_{xy}(t) = t^{-{\rm dim}(x \cap y)} (1-f_{W^-(x) \cap W^-(y)}(t)) $$
which has rational expressions in $t$. Define the Green function matrix which is a polynomial in $t$: 
$$  g_{xy}(t) = \omega(x) \omega(y) (1-f_{W^+(x) \cap W^+(y)}(t))  \; . $$
These definitions are triggered by the case $t=-1$, where we have seen that the inverse of 
$L=L_{xy}(-1)$, a matrix for which $L_{xy}=1$ if $x$ intersects with $y$ and $L_{xy}=0$ else. 

\paragraph{}
\begin{thm}[Determinant]
${\rm det}(L_t) =(-1)^{f_G(1)-1} t^{f_G'(1)} = \prod_x (-t)^{|x|}$. 
\end{thm}
\begin{proof}
The result holds in general for discrete CW complexes. Proceed by induction on the number of cells.
Every time we add a cell, the determinant gets multiplied by $(-t)^{|x|}$.
\end{proof}

\paragraph{}
There are two values, $t=1$ and $t=-1$ for which the matrix $L_{xy}(t)$ is unimodular. For $t=-1$, the 
determinant of $L_{-1}$ is the Fermi number $(-1)^f=\prod_{x \in G} \omega(x)$, 
where $f$ is the number of odd dimensional simplices. For $t=1$, the determinant of $-L_1$ is
$\prod_{x \in G} 1 = 1$. 

\begin{thm}[Green Star]
$g_t = L_t^{-1}$.
\end{thm}
\begin{proof}
See the computation before.
\end{proof} 

\paragraph{}
The $f$-function of a complex $G$ is defined as $f_G(t) = 1+f_0 t + \cdots +f_d t^{d+1}$. 
We have

\begin{thm}[Energy theorem]
$\sum_{x,y} g_{t}(x,y) = 1-f_G(t)$ for all $t$. 
\end{thm}
\begin{proof}
The parametrized Poincar\'e-Hopf theorem tells that 
$1-f_G(t) = -t \sum_{x} f_{S_g(x)}(t)$. 
Define the potential 
$k(x) =\sum_{y} g_{t}(x,y)$. It is enough to show that 
$k(x)=-t f_{S_g(x)}(t)$ for some function $g$. It seems
that $g=-{\rm dim}$ seems to work. 
\end{proof}

\paragraph{}
The case $t=1$ is of interest because, then $L$ is conjugated to its inverse $g$. 
The matrices $L_1,g_1 \in SL(n,\mathbb{Z})$ are isospectral and negative definite.
Let $\lambda_k$ be the eigenvalues of $-L_1$ which are also the eigenvalues of $-g_1$. 
Define the parametrized zeta-function of the complex $G$ as 
$$ \zeta_t(s) = \sum_{k=1}^n \lambda_k^{-s}  \; . $$
It is an entire function from $\mathbb{C} \to \mathbb{C}$ and unambiguously defined as
$\lambda_k^{-s}= e^{-\log(\lambda_k) s}$ with $\lambda_k>0$.  Because $\lambda(s) = \lambda(\overline{s})$
one has $\zeta_1(a+ib) = \zeta_1(-a+ib)$. 

\section{Representation}

\paragraph{}
The disjoint union of simplicial complexes forms a monoid which completes to a
group. Given an energy function $h$ on $G$ and an energy function $k$ on $H$,
there is a natural energy function $h_{G+H}$ on $G+H$ defined as $h_{G+H}(x)=h(x)$ if $x \in G$
and $(h+k)(x)=k(x)$ if $x \in H$. To extend the energy to the Grothendieck completion
define $h_{-G}(x) = -h_G(x)$. The Cartesian product $G*H$ has the energy $h_{G*H}(x \times y) = h_G(x) h_H(y)$. 
The total energy functional $E: G \to E[G] = \sum_{x \in G} h(x)$ is now a ring homomorphism almost
by definition.

\paragraph{}
To stay within a finite frame work assume $h$ to be integer valued also to stay closer to the
story to divisors. In the continuum, an integer valued function with finite
total energy $\sum_{x \in N} h(x)$ must to have a finite set as support, which means it is a divisor. 
With the disjoint union and Cartesian product, the class $\mathcal{X} = \{ X=(G,h) \}$ of 
simplicial complexes is a ring $\mathcal{X}$ which naturally extends the ring $\mathcal{G}$ of
simplicial complexes. The point is that we think of the energized complex as a geometric object
similarly as in the continuum, a vector bundle is a geometry object. 

\begin{propo}
a) The map $X=(G,h) \to G$ from $\mathcal{X}$ to $\mathcal{G}$  is a ring homomorphism. \\
b) The map $X=(G,h) \to E[X] = \sum_x h(x)$ from $\mathcal{X} \to \mathbb{Z}$ are ring homomorphisms.
\end{propo}

\paragraph{}
The prototype example is $h(x) = \omega(x)$ where the energy $E[G] = \chi(G)$ is the 
Euler characteristic of $G$. An other example is $h(x) = 1$, in which case $E[G] = |G|$ is the
number of elements. A third example is $h(x) = t^{|x|}$ in which case the total 
energy is a ring homomorphism from $\mathcal{N}$ to the polynomial ring $\mathcal{R}[t]$.
Energy functions $h$ of the form $h(x)=H(|x|)$ 
with multiplicative $H(n*m)=H(n) H(m)$ now defines a representation in a 
tensor ring of matrices. Examples are $h(x)=\omega(x)$ or $h(x)=1$ or $h(x)=|x|$. 

\paragraph{}
Now, if $h$ is an arbitrary energy function on complexes, then with the extension
$h_{G*H}(x \times y) = h_G(x) h_H(y)$, we have multiplicity and so:

\begin{thm}[Representation]
The map which assigns to $X=(G,h) \in \mathcal{X}$ the matrix $K$ 
is a ring homomorphism from $\mathcal{X}$ to the tensor ring of finite
matrices. 
\end{thm} 

\paragraph{}
One can assign to $X=(G,h)$ also a graph divisor $(\Gamma=(V,E),h)$, where 
$V=G,E=\{ (a,b) \; | \; a \cap b \neq \emptyset \}$. The graph $\Gamma$ is the 
connection graph of $G$. The addition is the disjoint
union still and the multiplication is the Sabidussi muliplication. Also this
is a ring. The graph complement operation maps this ring into subring of the
Zykov ring of graphs with join and dual Sabidussi multiplication, but now also
equipped with the energy. This is an interesting link as the analysis shows that
on the geometric side, there is a natural norm $| \cdot |$, which is 
the independence number of $\Gamma$. The integer ring $\mathcal{X}$ is therefore 
an arithmetic object with a norm $|(X,h)| = \sqrt{|X|^2+|h|^2}$ satisfying the inequality 
$|(X*Y,h*k)| \leq |(X,h)| |(Y,k)|$ and therefore can as $\mathcal{G}$ already be
extended to a Banach algebra $\mathcal{R}$.

\section{Gauss-Bonnet}

\paragraph{}
The classical Gauss-Bonnet theorem writes the Euler characteristic as a sum of curvatures.
It generalizes to valuations $X$, real valued maps from sub structures having the property 
$X(G \cup H) + X(G \cap H) = X(G) + X(H)$, but it holds also in a Hamiltonian 
set-up. If $\Gamma$ is a graph with Whitney complex $G$ equipped with a Hamiltonian $h$
define the parametrized energy $E_G(t)= \sum_{x \in G} h(x) t^{|x|}$ which is for $t=1$ 
the energy. Now $K(v) = \int_0^1 E_{S(v)}(s) \; ds$ defines a curvature for the graph and the 
energy can be written as

\begin{propo}
$E(G) = \sum_{v} K(v)$. 
\end{propo} 
\begin{proof}
A simplex of cardinality $|x|$ distributes the energy equally to $|x|+1$ points so 
that each point gets $h(x)/(|x|+1)$ which is $\int_0^1 h(x) t^{|x|} dt$. 
\end{proof}

\paragraph{}
Also here, we note already that this does simple principle does not even require the simplicial complex
structure. We just have to define $|x|$ in general as the number of atoms in $x$, 
where an ``atom" is a set in $G$ which does not have a proper non-empty subset. 
The curvature will then be supported on atoms of the structure. 

\paragraph{}
Poincar\'e-Hopf deals with a locally injective function $g$ on 
$G$. Define $S_g(v) = \{ y  \in G \; | \; g(y)<g(v) \}$. Now, we also have 
additionally an energy funtion $h$ given. This function $h$ does not need
to have any thing to do with $g$. But $h$ could be the index of $g$.

\begin{propo}
$E_G = \sum_{v \in V} E_{S_g(v)}$,
\end{propo}

\paragraph{}
We will discuss this a bit more in a follow-up. The upshot is that 
both Gauss-Bonnet, as well as Poincar\'e-Hopf hold not only for 
simplicial complexes but in a rather general set-up of energized
sets of sets. 

\section{Remarks} 

\paragraph{}
We have seen already that many of the results can be adapted to the situation where $G$ is just a set of finite non-empty sets. 
This is especially true for results about determinant and spectrum. Having a more symmetric
category, where one can also take the Boolean dual $\hat{G}$ of $G$ produces more simmetry. 
The matrices $L^{++}$ and $L^{--}$ are then completely dual
each other. Of course, there is still a reason to look at simplicial complexes. 
One of them is cohomology, an other is to be closer to classical geometries and 
especially manifolds. The escape to the larger category of sets of sets is illuminating however.

\paragraph{}
The nomenclature $L^{++} = E[W^+ \cap W^+]$, and 
$L^{--} = E[W^- \cap W^-]$ and $L^{+-}(x) = E[W^+(x) \cap W^-(y)]$
and $L^{-+} = E[W^-(x) \cap W^+(y)]$ is more symmetric. 
In the topological case $h(x)=\omega(x)$, the matrix $w=S L S$ is of interest
as $\sum_{x,y} w(x,y) = \omega(G)$ is the Wu characteristic of the complex
defined by $\sum_{x \sim y} \omega(x) \omega(y)$ summing over intersecting pairs $x \sim y$.
This suggests to look in general at the sum $W[G]=\sum_{x,y} S L S(x,y)$. 
This is the symmetric analogue of the sum $E[G] = \sum_{x,y} g(x,y)$. 
We have not yet investigated this. A first interesting case is in the constant
energy case $h(x)=1$, where $W[G] = \sum_{x,y} \omega(x) \omega(y) (2^{|x \cap y|}-1)$.

\paragraph{}
We could also look at the polynomial version $-g_{xy}(t)/t^{dim(x)}$ which is not symmetric
but has constant determinant $1$ and 
has the rational inverse $L_{xy}(t) = ((1+t)^{d+1} - t^{d+1})/t^d$, where $d={\rm dim}(x \cap y)$. 
For $t=0$, these are triangular matrices.  

\paragraph{}
The palindromic coefficient list $p_k$ of the characteristic polynomial ${\rm det}(g-\lambda)$ 
in the constant energy case $h(x)=1$ is very regular. If we plot the log of the 
coefficients, it approaches a parabola. But this regularity is only present
if we chose constant energy values. Already if we replace $h(x)=1$ with $h(x)=\omega(x)$, the 
parabola is less smooth. 

\paragraph{}
The choice of using integers $h(x)$ as energy values has reasons. The assumption is not necessary 
for the theorems covered here. But it allows to see $h$ as a divisor, which is an
integer valued function on a geometry with a finite set as support, the prototype is a divisor of a rational
function on a variety which as principal divisors define an equivalence class of all divisors.
Riemann-Roch expresses the energy $E(G)$ as a signed distance to $\{ E(G)=\chi(G) \}$.
In the continuum, where the base energy $E[X]=\omega(x)$ is included, the energy is written as
$\chi(G) + {\rm deg}(D)$. In the discrete, where one can chose space itself is part of the divisor
(in the form $h(x)=\omega(x)$ with total energy the Euler characteristic) and
rather than writing $l(D)-l(K-D)$, we just write $l(D)-l(-D)$ and
see $l(D)$ and $l(-D)$ as the analogous to $l(x)={\rm max}(x,0)$ and $l(-x)={\rm max}(-x,0)$ which
gives $|x|=l(x)-l(-x)$ which is Riemann-Roch for a 1-point space.

\paragraph{}
It is the privilege of finite geometries that we can chose a ``clean slate" zero energy $K=0$
as the canonical divisor $K$. In the continuum, (where we can not access lower dimensional ``atoms of
space" as points, we need sheave theoretical constructs like differential forms)
this is not possible and we need a canonical divisor (a global meromorphic function which has
as the degree the Euler characteristic of the curve).
Furthermore, in the continuum, the discreteness of energy forces an energy
function to be located on a finite set of points and where writing the Euler characteristic as an
energy of a natural divisor. In the finite, the geometry is implemented in the form of the divisor
$h(x)=(-1)^{{\rm dim}(x)}$ which assigns to a set an energy so that the total energy is
Euler characteristic.

\paragraph{}
Alternatively, one can think of $h(x)$ as an energy excitement level of a quantum harmonic oscillator.
While in physics, one might prefer $1/2 + h(x)$ with integer $h(x)$, a geometer would just write
$h(x) = \omega(x) + D(x)$ so that the energy
is $\chi(G) + {\rm deg}(D)$ as in Riemann-Roch. That theorem motivates to see
a divisor (similarly as a vector bundle) as a geometry by itself and extend the category of
simplicial complexes to simplicial complexes which are ``energized".
Essentially everything done in geometry works in the energized frame work. Examples
are Gauss-Bonnet, Euler-Poincar\'e. Riemann-Roch is a form of Euler-Poincar\'e for a
cohomology. Changing the divisor by adding principal divisors behaves like homotopy deformations.
The dimension value $l(G)$ appearing in Riemann Roch is a homotopy invariant in this picture.

\section{Code}

\paragraph{}
The following Mathematica code generates the matrix $K$ and its inverse $K^{-1}$,
and the zeta function for a random complex according to the definitions and
illustrates the results in examples. As usual, the code can be grabbed from
the ArXiv. It should serve as pseudo code also:

\begin{tiny}
\lstset{language=Mathematica} \lstset{frameround=fttt}
\begin{lstlisting}[frame=single]
Generate[A_]:=Delete[Union[Sort[Flatten[Map[Subsets,A],1]]],1];
RandomSets[n_,m_]:=Module[{A={},X=Range[n],k},Do[k:=1+Random[Integer,n-1];
  A=Append[A,Union[RandomChoice[X,k]]],{m}];Sort[Generate[A]]];
G=RandomSets[6,9];n=Length[G];    e=Table[RandomChoice[{1,-1}],{k,n}];
S=Table[-(-1)^Length[G[[k]]]*If[k ==l,1,0],{k,n},{l,n}]; (* Super matrix   *)
energy[A_]:=If[A=={},0,Sum[e[[Position[G,A[[k]]][[1,1]]]],{k,Length[A]}]];
star[x_]:=Module[{u={}},Do[v=G[[k]];If[SubsetQ[v,x],u=Append[u,v]],{k,n}];u];
core[x_]:=Module[{u={}},Do[v=G[[k]];If[SubsetQ[x,v],u=Append[u,v]],{k,n}];u];
Wminus     = Table[Intersection[core[G[[k]]],core[G[[l]]]],{k,n},{l,n}];
Wplus      = Table[Intersection[star[G[[k]]],star[G[[l]]]],{k,n},{l,n}];
Wminusplus = Table[Intersection[star[G[[k]]],core[G[[l]]]],{k,n},{l,n}];
Wplusminus = Table[Intersection[core[G[[k]]],star[G[[l]]]],{k,n},{l,n}];
Lminus     = Table[energy[Wminus[[k,l]]],    {k,n},{l,n}];  L  =      Lminus;
Lplus      = Table[energy[Wplus[[k,l]]],     {k,n},{l,n}];  g  =   S.Lplus.S;
Lminusplus = Table[energy[Wminusplus[[k,l]]],{k,n},{l,n}];
Lplusminus = Table[energy[Wplusminus[[k,l]]],{k,n},{l,n}];
Potential  = Table[Sum[g[[k,l]],{l,n}],{k,n}];
Superdiag  = Table[S[[k,k]]*g[[k,k]],{k,n}];
Print[{
Total[Flatten[g]] == Total[e],                 (* Potential energy         *)
Tr[S.g] == Total[e],                           (* Spectral energy          *)
Det[g]==Det[L]==Product[e[[k]], {k, n}],       (* Unimodularity Theorem    *)
Sort[Sign[Eigenvalues[1.0 L]]]==Sort[Sign[e]], (* eigenvalues of L         *)
Sort[Sign[Eigenvalues[1.0 g]]]==Sort[Sign[e]], (* eigenvalues of g         *)
Superdiag == Potential,                        (* selfinteraction=potential*)
g.L ==IdentityMatrix[n]}]                      (* only true in spin case   *)
\end{lstlisting}
\end{tiny}

\paragraph{}
Here is the code to draw the multi-graphs $\Gamma^{++}, \Gamma^{--}$ defined by their
adjacency matrices $L^{--},L^{++}$. In this example, $G$ is not a 
simplicial complex and the graphs have the adjacency matrices
$$ L^{--} = \left[ \begin{array}{ccccc}
 5 & 2 & 2 & 1 & 1 \\
 2 & 2 & 0 & 1 & 0 \\
 2 & 0 & 2 & 0 & 1 \\
 1 & 1 & 0 & 1 & 0 \\
 1 & 0 & 1 & 0 & 1 \\ \end{array} \right], 
   L^{++} = \left[ \begin{array}{ccccc}
 1 & 1 & 1 & 1 & 1 \\
 1 & 2 & 1 & 2 & 1 \\
 1 & 1 & 2 & 1 & 2 \\
 1 & 2 & 1 & 3 & 1 \\
 1 & 1 & 2 & 1 & 3 \\ \end{array} \right] \; . $$

\begin{tiny}
\lstset{language=Mathematica} \lstset{frameround=fttt}
\begin{lstlisting}[frame=single]
G = {{1,2,3}, {1,2}, {2,3}, {1}, {3}}; n = Length[G]; e = {1, 1, 1, 1, 1};
energy[A_]:=If[A=={},0,Sum[e[[Position[G,A[[k]]][[1,1]]]],{k,Length[A]}]];
star[x_]:=Module[{u={}},Do[v=G[[k]];If[SubsetQ[v,x],u=Append[u,v]],{k,n}];u];
core[x_]:=Module[{u={}},Do[v=G[[k]];If[SubsetQ[x,v],u=Append[u,v]],{k,n}];u];
Wminus     = Table[Intersection[core[G[[k]]],core[G[[l]]]],{k,n},{l,n}];
Wplus      = Table[Intersection[star[G[[k]]],star[G[[l]]]],{k,n},{l,n}];
Lminus     = Table[energy[Wminus[[k,l]]],    {k,n},{l,n}]; 
Lplus      = Table[energy[Wplus[[k,l]]],     {k,n},{l,n}];
U=Graph[Flatten[Table[Table[G[[k]]<->G[[l]],{Lminus[[k,l]]}],{k,n},{l,k,n}]]];
V=Graph[Flatten[Table[Table[G[[k]]<->G[[l]],{Lplus[[k,l]]}] ,{k,n},{l,k,n}]]];
DrawGraph[s_]:=Graph[s,GraphStyle->"SmallNetwork"]; 
Print["Isospectral:    ", Eigenvalues[Lminus] == Eigenvalues[Lplus]];
{DrawGraph[U],DrawGraph[V]}
\end{lstlisting}
\end{tiny}

\begin{figure}[!htpb]
\scalebox{0.8}{\includegraphics{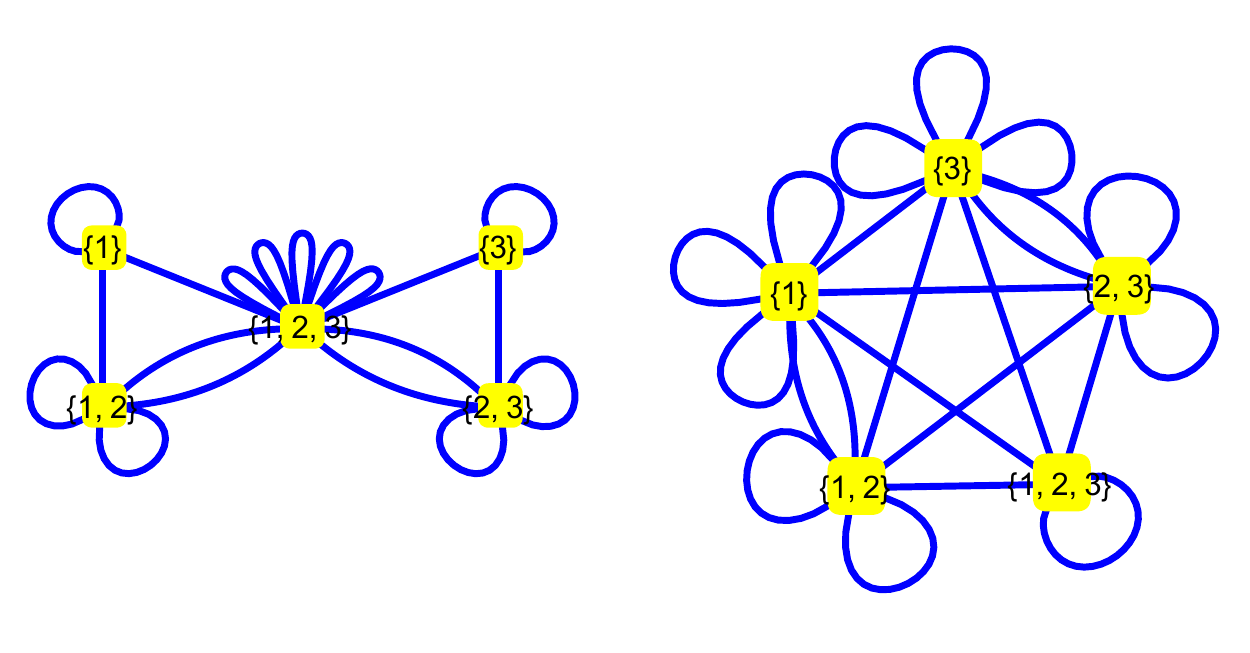}}
\label{example}
\caption{
We see the isospectral multi-graphs $\Gamma^{++},\Gamma^{--}$ produced in the code example,
where $G=\{\{1,2,3\},\{1,2\},\{2,3\},\{1\},\{3\}\}$, which is a
set of sets and not a simplicial complex. In the first case, two
nodes are connected by links encoding a common subset, in the
second case, two nodes are connected by links encoding a common superset. 
}
\end{figure}

\paragraph{}
And here is the parametrized version where a set with $k$ elements gets the energy $t^k$.
In this case, we scale the entries of $L^{--}$ a bit 
so that it matches with the inverse of $L^{++} = g$. The energy of a set of sets
is now the genus function $1-f_A(t)$. We definitely need a simplicial
complex for the parametrized result. 
The function $F[A]$ in the code is now the $f$-vector $f$ of $A$
and $f_A(t) = 1-\sum_k f_k(A) t^k$ so that $1-f_A(t)= \sum_k f_k(A) t^k$. 
This is equivalent to assigning the energy $t^k$ to a set with $k$ elements. 

\begin{tiny}
\lstset{language=Mathematica} \lstset{frameround=fttt}
\begin{lstlisting}[frame=single]
Generate[A_]:=Delete[Union[Sort[Flatten[Map[Subsets,A],1]]],1]; Clear[t]; 
RandomSets[n_,m_]:=Module[{A={},X=Range[n],k},Do[k:=1+Random[Integer,n-1];
  A=Append[A,Union[RandomChoice[X,k]]],{m}];Sort[Generate[A]]];
F[A_]:=If[u=Map[Length,A];A=={},A,Table[Length[Position[u,k]],{k,Max[u]}]];
f[A_,t_]:=Module[{R=F[A],k},If[A=={{}},1,1+Sum[R[[k]]*t^k,{k,Length[R]}]]];
G=RandomSets[4,6]; G=Generate[{{1,2},{2,3}}]; n=Length[G];  
S=Table[-(-1)^Length[G[[k]]]*If[k ==l,1,0],{k,n},{l,n}];energy[A_]:=1-f[A,t];
star[x_]:=Module[{u={}},Do[v=G[[k]];If[SubsetQ[v,x],u=Append[u,v]],{k,n}];u];
core[x_]:=Module[{u={}},Do[v=G[[k]];If[SubsetQ[x,v],u=Append[u,v]],{k,n}];u];
Wminus     = Table[Intersection[core[G[[k]]],core[G[[l]]]],{k,n},{l,n}];
Wplus      = Table[Intersection[star[G[[k]]],star[G[[l]]]],{k,n},{l,n}];
M=Table[Intersection[G[[k]],G[[l]]],{k,n},{l,n}]; 
Lminus     = Table[energy[Wminus[[k,l]]],    {k,n},{l,n}]; 
Lplus      = Table[energy[Wplus[[k,l]]],     {k,n},{l,n}];
g=S.Lplus.S;  L=Table[1-f[Wminus[[k,l]],t]/t^Length[M[[k,l]]],{k,n},{l,n}];
Potential  = Table[Sum[g[[k,l]],{l,n}],{k,n}];
Superdiag  = Table[S[[k,k]]*g[[k,k]],{k,n}];
Print[{
Total[Flatten[g]] == energy[G],               (* Potential energy         *)
Simplify[Tr[S.g]  == energy[G] ],             (* Spectral energy          *)
Det[g]==(-1)^(1-f[G,1])t^Total[Map[Length,G]],(* Unimodularity Theorem    *)
Superdiag == Potential,                       (* selfinteraction=potential*)
Simplify[g.L ==IdentityMatrix[n]]}]           (* g and L are inverse      *)
\end{lstlisting}
\end{tiny}

\paragraph{}
In the given case, where $G=\{\{1\},\{2\},\{3\},\{1,2\},\{2,3\}\}$, we have
$$ g = \left[ \begin{array}{ccccc}
                   -t^2-t & -t^2 & 0 & t^2 & 0 \\
                   -t^2 & -2 t^2-t & -t^2 & t^2 & t^2 \\
                   0 & -t^2 & -t^2-t & 0 & t^2 \\
                   t^2 & t^2 & 0 & -t^2 & 0 \\
                   0 & t^2 & t^2 & 0 & -t^2 \\
                  \end{array} \right]  \;  $$
and 
$$ L = \left[ \begin{array}{ccccc}
    1-\frac{t+1}{t} & 0 & 0 & 1-\frac{t+1}{t} & 0 \\
    0 & 1-\frac{t+1}{t} & 0 & 1-\frac{t+1}{t} & 1-\frac{t+1}{t} \\
    0 & 0 & 1-\frac{t+1}{t} & 0 & 1-\frac{t+1}{t} \\
    1-\frac{t+1}{t} & 1-\frac{t+1}{t} & 0 & 1-\frac{t^2+2 t+1}{t^2} &
      1-\frac{t+1}{t} \\
    0 & 1-\frac{t+1}{t} & 1-\frac{t+1}{t} & 1-\frac{t+1}{t} & 1-\frac{t^2+2
      t+1}{t^2} \\
   \end{array} \right]  \; . $$
The sum over all matrix elements $\sum_{x,y} g(x,y)$ in $g$ is $1-f_G(t) = -3t-2t^2$
which agrees with the super trace of $g$. 

\section{Some questions}

\paragraph{}
By choosing an energy function $h(x)$ with $k$ neegative values $-1$ and $n-k$ positive 
values $1$ produces non-negative symmetric $n \times n$ integral matrices $L$ with $L^{-1}$ integral
valued of that $L$ has $k$ negative and $n-k$ positive eigenvalues. This solves 
inverse problems for non-negative matrices like for example 
Corollary 2.1 in \cite{MincNonnegative}. 
Can one characterize the spectra of $n \times n$ matrices $L=L^{--}$ which 
appear for a set of $n$ sets $G$ and a function $h \in \{-1,1\}^n$? An inverse problem
is to get back $G$ from the spectrum. An easier task might be to reconstruct 
$G$ from knowing all the spectra for all possible energy functions $h \in \{-1,1\}^n$. 

\paragraph{}
One can wonder for which set $G$ of sets the matrices $L^{++}$ and $L^{--}$ are
completely positive definite in the sense of \cite{BermanMonderer}:
a matrix is completely positive if $A=B B^T$, where $B$ is a non-negative matrix.
Given a simplicial complex $G$ with $n$ elements, the Hodge matrix $H=(d+d^*)^2$ by 
definition is completely positive because the Dirac operator $D=d+d^*$ is a $0-1$
matrix and so non-negative. For the connection Laplacians this is still unclear
as there is no cohomology yet associated (except in one dimensions).

\paragraph{}
An other question is the multiplicity $m(G)$ of the eigenvalue $1$ in $L=L^{--}$ in the case
of constant energy $1$. It often 
correlates with the Betti numbers. For simplicial complexes obtained by 
one-dimensional grid graphs for example, it 
is always $b_0+b_1$. For complete complexes $G$, we often have $m(K_n)=1$ but $m(K_4)=5$
and $m(K_6)=19$ and $m(K_8)=69$. Investigating the structure of the eigenvalues of $L^{--}(G)$ which 
are the eigenvalues of $L^{++}(G)$ is in general not yet done. 

\begin{figure}[!htpb]
\scalebox{0.5}{\includegraphics{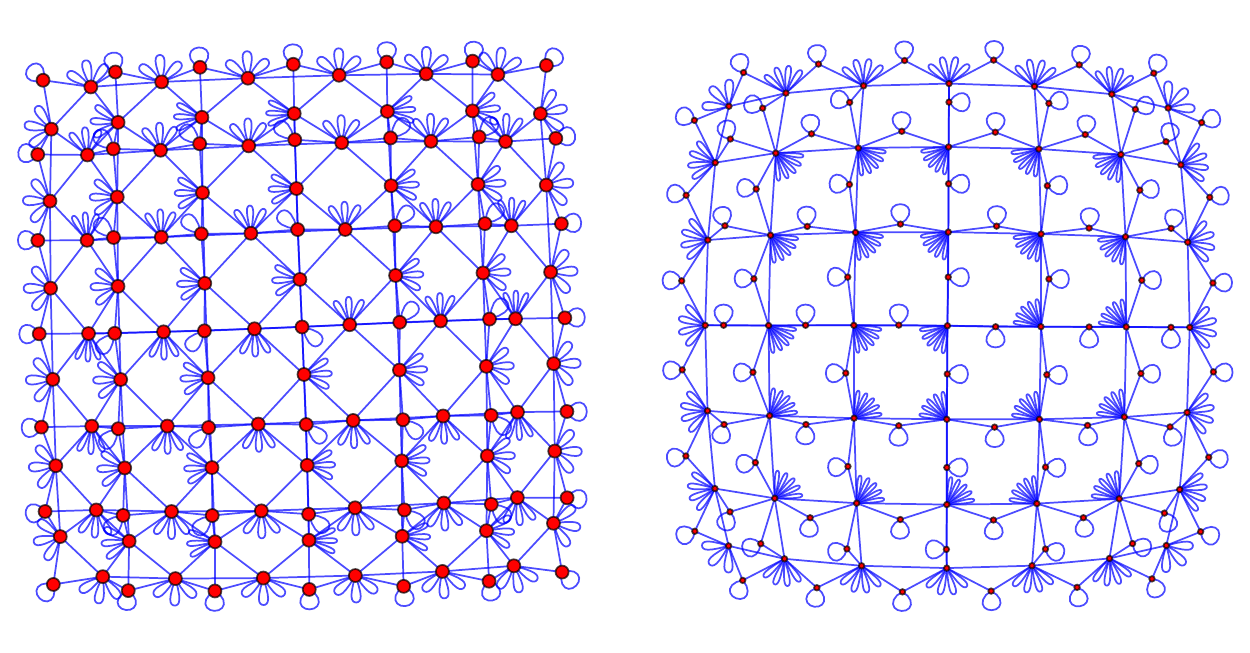}}
\scalebox{0.5}{\includegraphics{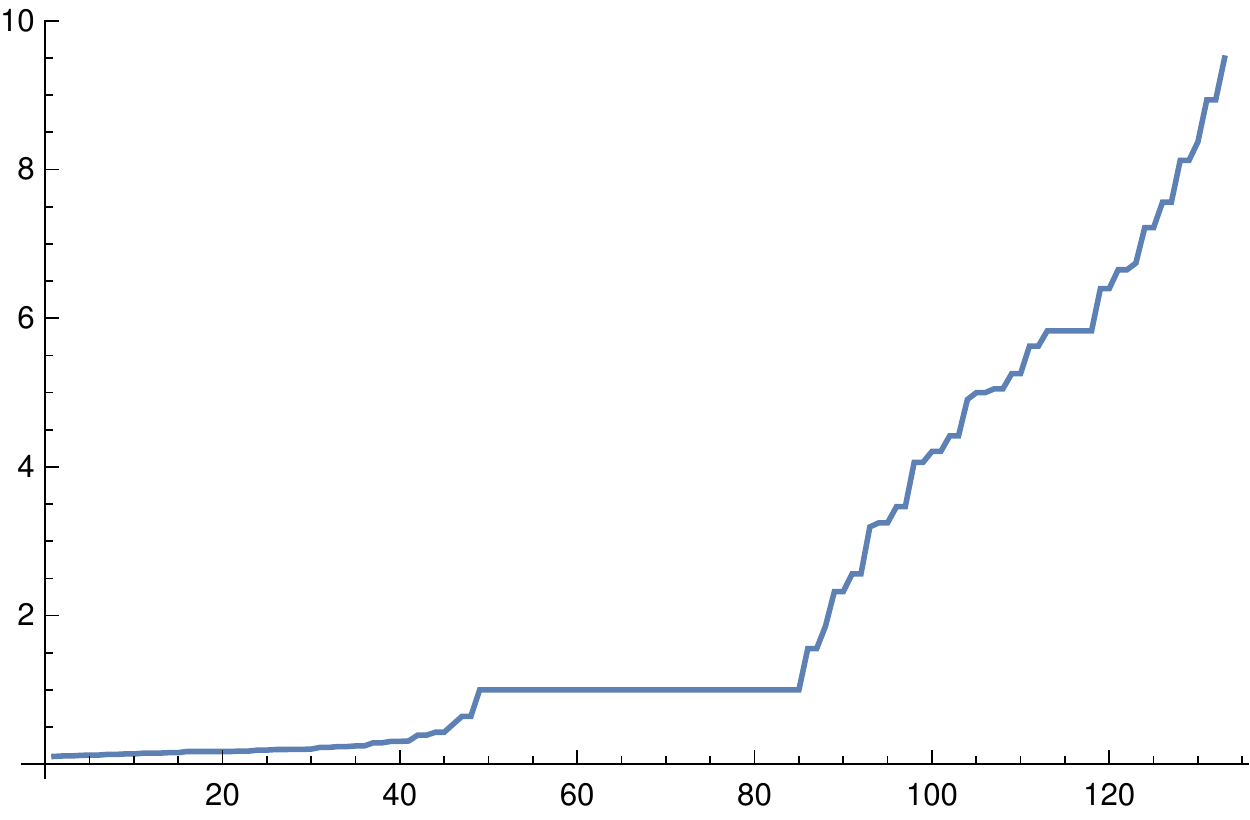}}
\label{grid}
\caption{
The Whitney complex $G$ of the grid graph $G(7,7)$ defines matrices $\Gamma^{++},\Gamma^{--}$
defined by the adjacency matrices $L^{++},L^{--}$. In the figure below we see the eigenvalues 
of $L^{++}$.  The flat large plateau consists of $37$ eigenvalues $1$. 
For grid graphs, we see that the number of eigenvalues $1$ is the sum of the 
Betti numbers of $G$. 
}
\end{figure}

\paragraph{}
The matrix $L$ defines a self-dual lattice packing
in $\mathbb{R}^n$. Every set of sets $G$ defines such a packing. The packing density is 
the same than than of the lattice $\mathbb{Z}^n$ as the minimal lattice distance is $1$. In the 
context of packings and more generally, for integral quadratic forms
one looks also at the $\theta$-function 
$$  \theta(z) = \sum_{m \in \mathbb{Z}^n} e^{2\pi i z Lm \cdot m} \; . $$
It is an other analytic way to describe the connection multi graph with adjacency matrix $L$ similarly as
the spectral zeta function 
$\zeta(s) = \sum_{k=1}^n \lambda_k^{-s}$ defined by the eigenvalues of $L$ 
or the Ihara zeta function $\zeta_I(s) = 1/\det(1-s L)$, a discrete analogue
of the Selberg zera function, do. 

\paragraph{}
It should be obvious that we get matrices $L^{++},L^{--}$ also if we start with 
differential complexes which are infinite. We can look at the infinite lattice $\mathbb{Z}^d$ for 
example and decorate various parts of the lattice to get periodic, almost periodic or random 
Schr\"odinger type matrices \cite{Cycon}. What remains true is that $L^{++}, L^{--}$ 
remain isospectral. While in the finite case, the matrices are diagonalizable, this is 
not necessarily the case in the infinite case as this requires a complete set of eigenfunctions
which one can not expect in general. 
In the periodic case for example the spectrum is continuous. In general, we expect
singular continuous spectrum to appear. All we know from general principles is that the 
density of states of $L^{++}$ and $L^{--}$ agree. A first thing to study is the one-dimensional 
Schr\"odinger case where we start with the simplicial complex of $\mathbb{Z}$. In that case
we know $L^{++},L^{--}$ in the limit as they are Barycentric limits which is completely understood
in one dimension but open in higher dimensions \cite{KnillBarycentric2}. 

\paragraph{}
In the periodic case $G=C_5$ for example with 
$$ G\{\{1\},\{1,2\},\{2\},\{2,3\},\{3\},\{3,4\},\{4\},\{4,5\},\{5\},\{5,1\}\} \; , $$
we have 
the isospectral periodic Schr\"odinger operators ``on a strip" (special T\"oplitz matrices): 
$$ L^{--} = \left[
\begin{array}{cccccccccc}
 1 & 1 & 0 & 0 & 0 & 0 & 0 & 0 & 0 & 1 \\
 1 & 3 & 1 & 1 & 0 & 0 & 0 & 0 & 0 & 1 \\
 0 & 1 & 1 & 1 & 0 & 0 & 0 & 0 & 0 & 0 \\
 0 & 1 & 1 & 3 & 1 & 1 & 0 & 0 & 0 & 0 \\
 0 & 0 & 0 & 1 & 1 & 1 & 0 & 0 & 0 & 0 \\
 0 & 0 & 0 & 1 & 1 & 3 & 1 & 1 & 0 & 0 \\
 0 & 0 & 0 & 0 & 0 & 1 & 1 & 1 & 0 & 0 \\
 0 & 0 & 0 & 0 & 0 & 1 & 1 & 3 & 1 & 1 \\
 0 & 0 & 0 & 0 & 0 & 0 & 0 & 1 & 1 & 1 \\
 1 & 1 & 0 & 0 & 0 & 0 & 0 & 1 & 1 & 3 \\
\end{array} \right] $$
and 
$$ L^{++} = \left[
\begin{array}{cccccccccc}
 3 & 1 & 1 & 0 & 0 & 0 & 0 & 0 & 1 & 1 \\
 1 & 1 & 1 & 0 & 0 & 0 & 0 & 0 & 0 & 0 \\
 1 & 1 & 3 & 1 & 1 & 0 & 0 & 0 & 0 & 0 \\
 0 & 0 & 1 & 1 & 1 & 0 & 0 & 0 & 0 & 0 \\
 0 & 0 & 1 & 1 & 3 & 1 & 1 & 0 & 0 & 0 \\
 0 & 0 & 0 & 0 & 1 & 1 & 1 & 0 & 0 & 0 \\
 0 & 0 & 0 & 0 & 1 & 1 & 3 & 1 & 1 & 0 \\
 0 & 0 & 0 & 0 & 0 & 0 & 1 & 1 & 1 & 0 \\
 1 & 0 & 0 & 0 & 0 & 0 & 1 & 1 & 3 & 1 \\
 1 & 0 & 0 & 0 & 0 & 0 & 0 & 0 & 1 & 1 \\
\end{array}
\right]  \; . $$
What happens here is that in the limit where the matrices are bounded operators on $l^2(\mathbb{Z})$, 
the two matrices are the same and invertible. Unlike in the usual discrete Schr\"odinger case (which 
is the Hodge case), where the spectral interval has $0$ at the boundary which leads in dynamical 
system cases to KAM situations as one needs then a strong implicit function theorem to 
perturb, the connection case with a mass gap is hyperbolic. 

\paragraph{}
In the above case, the isospectral property in the limit was trivial. If we take limits
of periodic graphs which are a bit fatter 

\begin{figure}[!htpb]
\scalebox{0.5}{\includegraphics{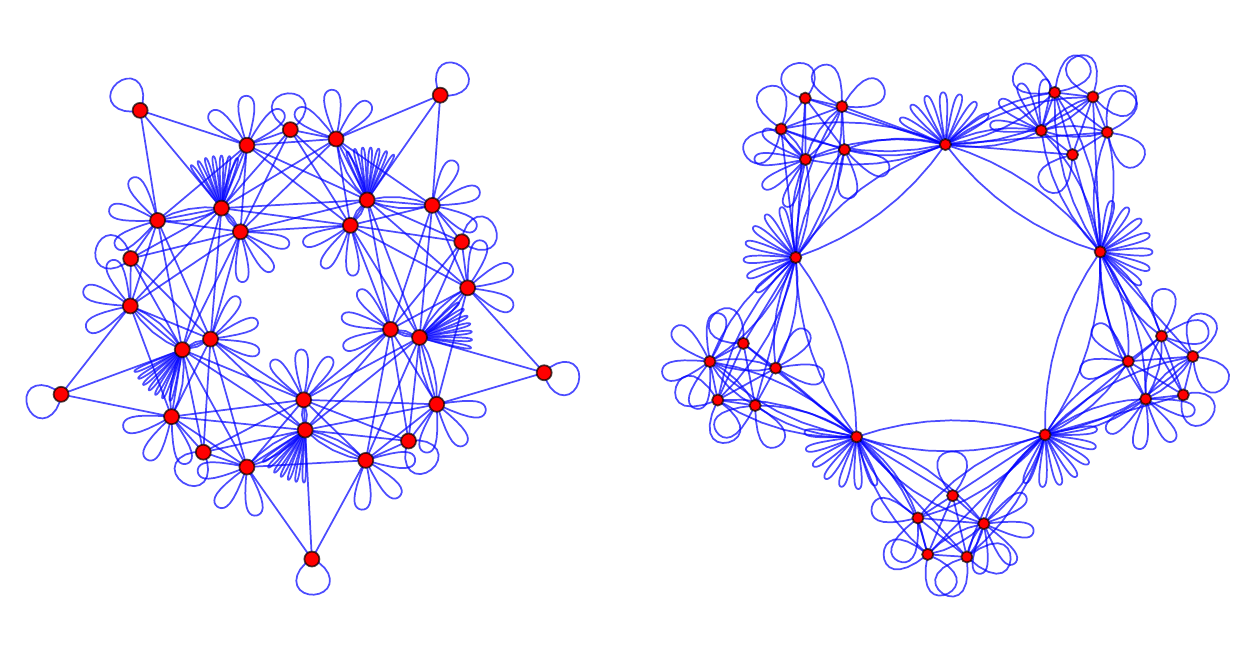}}
\label{artillery}
\caption{
Two isospectral periodic graphs $\Gamma^{++},\Gamma^{--}$ defined
by a simplicial complex from the connection graph of $C_{10}$ are
non-trivially isospectral in the limit. The Jacobi matrices 
$L^+$ and $L^-$ on $l^2(\mathbb{Z})$ have the same density of states.
}
\end{figure}

\paragraph{}
For which set of sets $G$ does the 15 theorem kick in? Here are two examples:
for the set of sets $G=\{ \{1,2\},\{1\} \}$ the ``komma" for which the connection Laplacian
for constant energy $h(x)=1$ gives the matrix
$L=\left[ \begin{array}{cc} 2 & 1 \\ 1 & 1 \end{array} \right]$ which is
the Arnold cat matrix, the quadratic form is $Q(x,y) = 2 x^2 + 2 x y + y^2$. 
It is not universal because the value $3$ is not attained. However, for 
$G= \{\{1\},\{2\},\{3\},\{1,3\}\}$, we have $L= \left[ \begin{array}{cccc}
                    1 & 0 & 0 & 1 \\
                    0 & 1 & 0 & 0 \\
                    0 & 0 & 1 & 1 \\
                    1 & 0 & 1 & 3 \\
                   \end{array} \right]$, which leads to a quadratic form
$Q(x,y,z,w)= 3 w^2+2 w (x+z)+x^2+y^2+z^2$ which is universal as all values from $\{1, \cdots, 15 \}$
are attained. Preliminary experiments show that only very small sets of sets lead to non-universal
quadratic forms. 

\bibliographystyle{plain}

\end{document}